\documentclass[a4paper,leqno,10pt]{amsart}

\usepackage[utf8]{inputenc}
\usepackage[T1]{fontenc}
\usepackage[english]{babel}

\usepackage{amsfonts}
\usepackage{amsmath}
\usepackage{amsthm}
\usepackage{amssymb}
\usepackage{indentfirst}
\usepackage{color}
\usepackage{fullpage}

\usepackage{mathrsfs}
\usepackage{mathtools}
\usepackage{enumitem}
\usepackage[normalem]{ulem}

\newcommand{\ubar}{\bar u}

\numberwithin{equation}{section}

\begingroup
\theoremstyle{plain}
\newtheorem{theorem}{Theorem}[section]

\newtheorem{proposition}[theorem]{Proposition}
\newtheorem{lemma}[theorem]{Lemma}
\theoremstyle{definition}
\newtheorem{definition}[theorem]{Definition}

\theoremstyle{remark}
\newtheorem{remark}[theorem]{Remark}
\endgroup

\theoremstyle{definition}
\theoremstyle{remark}

\newcommand{\eps}{\varepsilon} 
\newcommand{\weakto}{ \rightharpoonup}  
\newcommand{\stress}{\boldsymbol{\sigma}} \newcommand{\strain}{\boldsymbol{\epsilon}} 

\newcommand{\FF}{\mathbf{F}}
\newcommand{\tr}{\mathrm{tr}}
\newcommand{\EE}{\mathbf{E}}
\newcommand{\R}{\mathbb{R}}
\newcommand{\E}{\mathcal{E}}

\newcommand{\C}{\mathbb{C}}

\newcommand{\di}{\mathrm{d}}

\newcommand{\HH}{\mathcal{H}}
\newcommand{\Om}{\Omega}

\newcommand{\F}{\mathcal{F}}

\newcommand{\D}{\mathcal{D}}

\newcommand{\UU}{\mathcal{U}}

\newcommand{\coloneq }{:=}

\newcommand{\abs}[1]{\lvert #1 \rvert}

\newcommand{\norm}[1]{\lVert #1 \rVert}

\newcommand{\argmax}{\mathrm{argmax}}

\DeclareMathOperator*{\argmin}{arg\,min}

\title{Irreversibility and alternate minimization in phase field fracture: a viscosity approach}
\author[S. Almi]{Stefano Almi}
\address[Stefano Almi]{University of Vienna, Oskar-Morgenstern-Platz 1, 1090 Vienna, Austria}
\email{stefano.almi@univie.ac.at}

\begin{document}
\maketitle

\begin{abstract}
\noindent This work is devoted to the analysis of convergence of an alternate (staggered) minimization algorithm in the framework of phase field models of fracture. The energy of the system is characterized by a nonlinear splitting of tensile and compressive strains, featuring non-interpenetration of the fracture lips. The alternating scheme is coupled with an $L^{2}$-penalization in the phase field variable, driven by a viscous parameter~$\delta>0$, and with an irreversibility constraint, forcing the monotonicity of the phase field only w.r.t.~time, but not along the whole iterative minimization. We show first the convergence of such a scheme to a viscous evolution for $\delta>0$ and then consider the vanishing viscosity limit $\delta\to 0$.
\end{abstract}

{\small

\bigskip

\bigskip
\subjclass{\textbf{MSC 2010:} 
35Q74,    
49J45,     
74R05,    
74R10.	
\bigskip
\bigskip

\section{Introduction}


In the seminal work~\cite{MR1745759} the quasi-static propagation of brittle fractures in linearly elastic bodies is approximated in terms of equilibrium states of the \emph{Ambrosio-Tortorelli functional}
\begin{equation}\label{intro1}
\mathcal{G}_{\varepsilon}(u,z) \coloneq \tfrac{1}{2} \int_{\Om} (z^{2}+\eta_{\varepsilon}) \stress(u){\,:\,}\strain(u)\,\di x + G_{c} \int_{\Om} \varepsilon |\nabla{z}|^{2} + \tfrac{1}{4\varepsilon} (z-1)^{2}\,\di x \,, 
\end{equation}
where~$\Om$ is an open bounded subset of~$\R^{n}$ with Lipschitz boundary~$\partial\Om$, $u\in H^{1}(\Om;\R^{n})$ is the displacement,~$\strain(u)$ denotes the symmetric part of the gradient of~$u$,~$\stress(u):=\mathbb{C}\strain(u)$ is the stress,~$\mathbb{C}$ being the usual elasticity tensor,~$\varepsilon$ and~$\eta_{\varepsilon}$ are two small positive parameters, and~$G_{c}$  is the toughness, a positive constant related to the physical properties of the material under consideration (from now on we impose $G_{c}=1$). The so-called \emph{phase field} function~$z\in H^{1}(\Om)$ is supposed to take values in~$[0,1]$, where~$z(x)=1$ if the material is completely sound at~$x$, while $z(x)=0$ means that the elastic body~$\Om$ has developed a crack at~$x$. Hence, the zero level set of~$z$ represents the fracture and~$z$ can be interpreted as a regularization of the crack set. In the static framework, the connection between~\eqref{intro1} and fracture mechanics has been drawn in~\cite{MR1075076, MR1978582, MR3928751, MR3247391}, where the authors showed the $\Gamma$-convergence of~$\mathcal{G}_{\varepsilon}$ as~$\varepsilon\to 0$ to the functional
\begin{equation}\label{intro1.1}
\mathcal{G}(u) \coloneq \tfrac{1}{2} \int_{\Om}  \stress(u){\,:\,}\strain(u)\,\di x + \HH^{n-1}(J_u) \qquad\text{for $u\in GSBD^{2}(\Om;\R^{n})$}\,,
\end{equation}
where $\HH^{n-1}$ denotes the $(n-1)$-dimensional Hausdorff measure and~$J_u$ is the discontinuity set of~$u$.

From the computational point of view, the study of the functional~\eqref{intro1} is very convenient in combination with the so-called \emph{alternate minimization} algorithm \cite{MR3376787, MR2341850, MR1745759, Burke2010, Burke2013}: equilibrium configurations of the energy are computed iteratively by minimizing~$\mathcal{G}_{\varepsilon}$ first w.r.t.~$u$ and then w.r.t.~$z$. A fracture irreversibility is imposed by forcing~$z$ to be non-increasing in time. Exploiting the separate convexity of~$\mathcal{G}_{\eps}$, the above scheme guarantees the convergence to a critical point of~$\mathcal{G}_{\eps}$, whose direct computation would be rather time consuming because of the non-convexity of the functional.

 Let us turn our attention to the problem of evolution of the phase field driven by the energy functional~$\mathcal{G}_{\varepsilon}$. In the context of rate-independent processes~\cite{MR2182832, MielkeRoubicek}, the literature related to the general existence of such evolutions is very rich. The papers~\cite{Giacomini2005, MR2983477, MR2640367} showed the existence of~\emph{energetic solutions}, in which the equilibrium configurations are intended to be \emph{global minimizers} of~$\mathcal{G}_{\varepsilon}$. Global minimality is however expected to be unphysical in this context, since it generates time discontinuities in which the solution jumps between two equilibrium states, ignoring the presence of energetic barriers between them. For this reason, the works~\cite{MR3249813, MR3021776, MR3332887, MR3893258, MR3264231, Negri_ACV} promoted a vanishing viscosity approach, based on a viscous regularization of the evolution problem. As the viscosity tends to zero, the solution of the regularized problem converges to a \emph{Balanced Viscosity (BV)} evolution~\cite{MR2887927, MR3531671}, where the equilibrium states are \emph{critical points} of~$\mathcal{G}_{\eps}$. In all the mentioned papers, the proof of existence is very constructive and is based on a scheme that only partly resembles the alternate minimization algorithm used in computational fracture mechanics: besides a time-discretization procedure, the authors made use of a \emph{one-step scheme}, that is, at each time step the energy functional~$\F_{\eps}$ (or a suitable viscous perturbation) is minimized in the pair~$(u, z)$, and a time-continuous solution is obtained in the limit as the time step tends to zero. 

The theoretical investigation of the relationship between alternate minimization schemes for phase field models of fracture and rate-independent processes has instead only recently started with the works~\cite{Almi2017, MR3945577, A-N19, MR3669838} (see also~\cite{Knees2018, MR3899181} for the application of alternating algorithms in different physical frameworks). Here, we describe the algorithm adopted in~\cite{MR3669838}. In dimension $n=2$, given a time horizon~$T>0$, a time dependent Dirichlet boundary condition $t\mapsto g(t)$, and a suitable initial condition~$(u_0, z_0)$, the authors considered, as usual in the study of rate-independent systems, a time discretization procedure: for every $k\in\mathbb{N}$ set $\tau_{k}\coloneq T/k$ the time step increment and $t^{k}_{i}\coloneq i\tau_{k}$, $i=0, \ldots, k$, the time nodes. A time discrete evolution is then constructed through an iterative scheme: Known the state~$(u^{k}_{i-1}, z^{k}_{i-1})$ at time~$t^{k}_{i-1}$, we define $u^{k}_{i,0}\coloneq u^{k}_{i-1}$, $z^{k}_{i,0}\coloneq z^{k}_{i-1}$ and, for $j\geq1$,
\begin{eqnarray}
 &&\displaystyle u^{k}_{i,j}\coloneq \argmin \, \{\mathcal{G}_\eps (u,z^{k}_{i,j-1}):\, u\in H^{1}(\Om;\R^{2}),\, u= g (t^{k}_{i}) \text{ on~$\partial\Om$} \}\label{intro3}\,, \\
 &&\displaystyle z^{k}_{i,j} \coloneq \argmin \, \{ \mathcal{G}_\eps (u^{k}_{i,j},z) :\, z\in H^{1}(\Om),\, z \le z^k_{i,j-1} \} \label{intro4}\,.
\end{eqnarray}
In the limit $j\to\infty$, the algorithm~\eqref{intro3}-\eqref{intro4} identifies a critical point~$(u^{k}_{i}, z^{k}_{i})$ of~$\mathcal{G}_{\eps}$.

The analysis performed in~\cite{MR3669838} ensures that the above discrete solutions converge to a phase field evolution as the time increment~$\tau_{k}$ tends to zero. In particular, the limit solutions are characterized in terms of parametrized BV evolutions. We mention that a discrete version of~\cite{MR3669838} in a space-discrete (finite element) setting has been studied in~\cite{Almi2017} together with the limit of the solutions in a discrete to continuum sense, i.e., as the mesh becomes finer and finer. The results of~\cite{MR3669838} have been further generalized in~\cite{A-N19} to not separately quadratic energy functionals. An example of such energies, that will be considered also in this paper, is inspired by the phase field model introduced in~\cite{A-M-M09, MR3780140}. The underlying idea is that an elastic material behaves differently when under tension or compression. Moreover, fracture propagation is allowed only as a result of tensile or shear stresses, while compression does not lead to inelastic behaviors. Thus, in contrast with~\eqref{intro1}, the factor $z^{2}+\eta_{\varepsilon}$ shall not affect the whole stress~$\stress(u)$. Instead, a splitting of the strain~$\strain(u)$ into its volumetric~$\strain_{v} (u)\coloneq \tfrac{1}{2}(\tr \strain(u))\mathbf{I}$ and deviatoric $\strain_{d}(u)\coloneq \strain(u) - \strain_{v}(u)$ parts is considered, where~$\tr$ denotes the trace of a matrix and~$\mathbf{I}$ is the identity matrix. Introducing also tensile and compressive strains $\strain_{v}^{\pm}(u)\coloneq \frac{1}{2}(\tr \strain(u))_{\pm} \mathbf{I}$, the elastic energy density writes as
\begin{equation}\label{intro6}
W_\eps(z, \strain(u)) \coloneq h_{\varepsilon}(z)(\mu |\strain_{d}(u)|^{2} + \kappa |\strain_{v}^{+}(u)|^{2}) + \kappa |\strain_{v}^{-}(u)|^{2}
\end{equation}
and the phase field energy becomes
\begin{equation}\label{intro7}
\F_{\varepsilon}(z, u) \coloneq \int_{\Om} W_\eps (z, \strain(u)) \,\di x + \int_{\Om} ( |\nabla z |^{2} + f_\eps(z)) \, \di x \qquad\text{for $u\in H^{1}(\Om;\R^{2})$ and $z\in H^{1}(\Om)$}\,.
\end{equation}
In the above formulas, $\mu$ and $\kappa$ are two positive parameters related to the Lam\'e coefficients of the elastic material, and $h_\eps, f_\eps \colon \R\to[0,+\infty)$ are two \emph{degradation functions}. We refer to Section~\ref{setting} for the whole set of hypotheses and to~\cite{A-M-M09} for more details about the model. Here we only mention that in~\cite{MR3780140} it has been proven that, in dimension $n=2$,~$\F_\eps$ $\Gamma$-converges to~$\mathcal{G}$ in~\eqref{intro1.1} for $u \in SBD^{2}(\Om; \R^{2})$ satisfying the non-interpenetration constraint $[u]\cdot \nu_{u} \geq 0$, where $\nu_{u}$ is the approximate unit normal to~$J_u$ and~$[u]$ stands for the amplitude of the jump of~$u$ across~$J_u$. 

Despite the sound mathematical results obtained in~\cite{Almi2017, A-N19, MR3669838}, one of the drawback of the alternating scheme \eqref{intro3}-\eqref{intro4} lies in the irreversibility condition $z\leq z^{k}_{i,j-1}$, which forces the phase field variable to be non-increasing along the whole algorithm. This requirement, indeed, could lead to an accumulation of numerical error, making the fracture simulation very inaccurate. In order to bypass such a problem, the weaker irreversibility constraint $z\leq z^{k}_{i-1}$ is usually numerically imposed (see, for instance,~\cite{MR3376787, MR2341850, MR1745759, Burke2010, Burke2013}). From a theoretical viewpoint, very little is known about convergence of the scheme~\eqref{intro3}-\eqref{intro4} with this new irreversibility condition, mainly because the lack of monotonicity of~$z$ along the algorithm prevents from deducing any time regularity of the discrete evolutions, such as BV in time, and makes the analysis of~\cite{A-N19, MR3669838} out of reach. Up to our knowledge, the only existing result is~\cite{MR3945577}, where the minimization~\eqref{intro4} is replaced by
\begin{equation}\label{intro5}
z^{k}_{i,j}\coloneq\min\,\{\tilde{z}^{k}_{i,j}, z^{k}_{i-1}\}
 \quad \text{where} \quad  \tilde z^{k}_{i,j} \coloneq \argmin\,\{\mathcal{G}_\eps (u^{k}_{i,j},z) + \tfrac{1}{2 \tau_k} \| z - z^k_{i-1} \|^2_{L^2}:\, z \in H^{1}(\Om) \} \,.
\end{equation}
In particular, we notice that the functional~$\mathcal{G}_{\eps}$ is perturbed with an $L^{2}$-penalization, which makes the time discrete evolution regular in time. Furthermore, the minimum problem in~\eqref{intro5} is unconstrained, while the irreversibility is a posteriori imposed by truncating the minimizer~$\tilde{z}^{k}_{i,j}$ with~$z^{k}_{i-1}$. In the limit as $\tau_{k}\to 0$ we have been able to show the convergence to an $L^{2}$-gradient flow of~$\mathcal{G}_{\eps}$, while the presence of the pointwise minimization prevented us from studying the vanishing viscosity limit, and hence the convergence to a BV evolution.

The scope of this note is to give a first result of convergence of an alternate minimization scheme to a quasi-static evolution, in the presence of the (weaker) irreversibility constraint $z\leq z^{k}_{i-1}$. Precisely, with the notation introduced in~\eqref{intro6}-\eqref{intro7}, for every $\delta>0$ and $k\in\mathbb{N}$ we consider the iterative algorithm
\begin{eqnarray}
 &&\displaystyle u^{k}_{i,j} \coloneq \argmin\,\{\mathcal{F}_\eps (u,z^{k}_{i,j-1}):\, u\in H^{1}(\Om;\R^{2}),\, u= g (t^{k}_{i}) \text{ on~$\partial\Om$} \}\label{intro8}\,, \\
 &&\displaystyle z^{k}_{i,j} \coloneq \argmin\, \{ \mathcal{F}_\eps (u^{k}_{i,j},z) + \tfrac{\delta}{2\tau_{k}} \| z - z^{k}_{i-1} \|_{2}^{2} :\, z\in H^{1}(\Om),\, z \le z^k_{i - 1} \} \label{intro9}\,.
\end{eqnarray}
While~\eqref{intro8} is exactly as~\eqref{intro3}, we notice that~\eqref{intro9} is intermediate between~\eqref{intro4} and~\eqref{intro5}. Indeed, we have now explicitly imposed in the minimization the monotonicity constraint~$z \leq z^{k}_{i-1}$, which only ensures an irreversibility w.r.t.~time but not along the whole scheme. As in~\eqref{intro5}, we have regularized in time the evolution of the phase field~$z$ by adding to~$\F_\eps$ an $L^{2}$-penalization driven by a small viscous parameter~$\delta$.

Following the lines of~\cite{MR3454016, MR3021776}, in Sections~\ref{s.algorithm} and~\ref{s.evolution} we study the limit of~\eqref{intro8}-\eqref{intro9} as~$k\to\infty$ and $\delta\to 0$, in the given order. Hence, in Theorem~\ref{t.viscous} we show  for $\delta>0$ the convergence of the iterative scheme to a viscous evolution~$(u_{\delta}, z_{\delta})$ satisfying the following \emph{displacement equilibrium} and \emph{energy balance}:
\begin{eqnarray}
&& \displaystyle u_{\delta} (t) = \argmin \, \{ \F_{\eps}(u, z_{\delta}(t)) : \, u\in H^{1}(\Om;\R^{2}), \, u= g(t) \text{ on $\partial \Om$}\} \qquad \text{for every $t\in[0,T]$}\,, \label{intro10}\\[2mm]
&& \displaystyle \dot{\F}_{\eps}(u_{\delta}(t), z_{\delta}(t)) =  - \tfrac{1}{2\delta} |\partial_{z}^{-} \F_{\eps}|^{2} (u_{\delta} (t), z_{\delta} ( t) ) - \tfrac{\delta}{2} \| \dot{z}_{\delta} (t) \|_{2}^{2} + \partial_{t} \F_{\eps} ( u_{\delta} (t), z_{\delta}(t)) \qquad\text{for a.e.~$t\in [0,T]$}\,,
\end{eqnarray}
where $|\partial_{z}^{-}\F_{\eps}|$ denotes the \emph{unilateral} slope of~$\F_{\eps}$ and the dot indicates the derivative w.r.t.~time. We refer to Definitions~\ref{definitionSlope} and~\ref{d.Evolution} for the full details.

Finally, in Theorem~\ref{t.vanevolution} we consider the vanishing viscosity limit $\delta\to 0$ and prove that the pair~$(u_{\delta}, z_{\delta})$ converges, in a time reparametrized setting, to a BV evolution, now represented by a triple $(t, u, z)$, where~$t$ is a suitable Lipschitz parametrization of the time interval~$[0,T]$. Besides the  equilibrium condition~\eqref{intro10}, the triple~$(t, u, z)$ also satisfies the energy balance
\begin{displaymath}
\F_{\eps}'(u(s), z(s)) = -|\partial_{z}^{-}\F_{\eps}| (u(s), z(s)) \| z'(s) \|_{2} + \partial_{t} \F_{\eps} (u(s), z(s)) t' (s) \qquad\text{for a.e.~$s$}\,,
\end{displaymath}
where the index $'$ denotes the derivative w.r.t.~the new time variable~$s$. As in~\cite{MR3454016, MR3021776}, the time reparametrization is based on a uniform estimate of the length of the curve $t\mapsto z_{\delta}(t)$. Namely, in Lemma~\ref{l.length} we prove that the arc length of the algorithm~\eqref{intro8}-\eqref{intro9}, measured in terms of the distance between two consecutive states of the iterative minimization, is bounded uniformly w.r.t.~$k$ and~$\delta$. Eventually, this allows us to perform an arc length reparametrization~$s\mapsto t_{\delta}(s)$ of time which makes~$(u_{\delta}, z_{\delta})$ 1-Lipschitz continuous, and thus compact, in the new time variable. We notice that the uniform bound of the arc length is  a consequence of the combination of a general regularity and continuity result~\cite{HerzogMeyerWachsmuth_JMAA11} (see also Lemma~\ref{l.regLp}) for PDEs with non-constant coefficients and of Sobolev embeddings, valid only in dimension two.


\section{Notation and setting of the problem}\label{setting}

In order to explain the phase field model considered in this work, we have to introduce some notation. Let~$\mathbb{M}^{2}$ denote the space of squared matrices of order~$2$ and~$\mathbb{M}^{2}_{s}$ be the subspace of symmetric matrices. For every $\FF\in\mathbb{M}^{2}$, we consider its splitting in volumetric and deviatoric part 
\begin{equation*} 
\FF_{v}\coloneq \tfrac{1}{2}(\tr \FF) \mathbf{I} \qquad \text{and} \qquad \FF_{d}\coloneq \FF-\FF_{v}\,.
\end{equation*}
Note that $\FF_{v}{\,:\,}\FF_{d} = 0$, where the symbol~$:$ denotes the scalar product between matrices. As a consequence, we have that
\begin{displaymath}
|\FF|^{2}=|\FF_{v}|^{2}+|\FF_{d}|^{2}\qquad\text{for every $\FF\in\mathbb{M}^{2}$}\,,
\end{displaymath}
where $| \cdot |$ denotes Frobenius norm. Moreover, we set 
\begin{displaymath}
\FF_{v}^{\pm}\coloneq \tfrac{1}{2}(\tr \FF)_{\pm} \mathbf{I} \,,
\end{displaymath}
where~$(\cdot)_{+}$ and~$(\cdot)_{-}$ stand for positive and negative part, respectively. It is clear that $| \FF_v |^2 = | \FF^+_v |^2 + | \FF^-_v |^2$. 

Given $\EE\in \mathbb{M}^{2}_{s}$, we can rewrite the usual linear elastic energy density as
\begin{displaymath}
\begin{split}
\C\EE : \EE & = \tfrac{\lambda}{2} |\tr\EE|^{2}+\mu|\EE|^{2} = \lambda(|\EE_{v}^{+}|^{2}+|\EE_{v}^{-}|^{2})+\mu(|\EE_{v}^{+}|^{2}+|\EE_{v}^{-}|^{2}+|\EE_{d}|^{2}) \nonumber\\
&=\mu|\EE_{d}|^{2}+\kappa|\EE_{v}^{+}|^{2} + \kappa|\EE_{v}^{-}|^{2}
\end{split}
\end{displaymath} 
where $\lambda, \mu$ are the Lam\'e coefficients and $\kappa\coloneq\lambda+\mu$ We assume $\mu, \kappa >0$.

Following the lines of~\cite{A-M-M09, MR3780140}, we consider a phase field model that does not allow for fracture under compression, that is, when $( \tr \EE)_{-} \neq 0$. To model such a behavior, the phase field variable~$z\in[0,1]$ is assumed to affect only the energetic contribution of the tensile strain~$\EE_{v}^{+}$ and of the deviatoric strain~$\EE_{d}$. Hence, the \emph{elastic energy} density reads 
\begin{displaymath}
W(z,\EE)\coloneq h(z)\big( \mu |\EE_{d}|^{2} + \kappa |\EE_{v}^{+}|^{2} \big) + \kappa |\EE_{v}^{-}|^{2} \qquad \text{for $z\in\R$ and  $\EE\in\mathbb{M}^{2}_{s}$} \,,
\end{displaymath}
where $h\colon \R \to[0,+\infty)$ is the {\it degradation function}. We assume that $h \in C^{1,1}_{loc}(\R)$ is convex and such that
\begin{displaymath}
 \text{ $h(z)\geq h(0) >0$ for every $z\in\R$.}
\end{displaymath}
Notice that, under these assumptions, $h$ is non-decreasing in $[0,+\infty)$. 

For $z$ fixed, the function $W ( z, \cdot) $ is differentiable w.r.t.~$\EE$ and 
\begin{displaymath}
\partial_{\mathbf{E}} W(z,\mathbf{E} ) = 2 h(z) \big(  \mu \mathbf{E}_{d} + \kappa \mathbf{E}_{v}^{+} \big) - 2\kappa \mathbf{E}_{v}^{-} \,.
\end{displaymath} 
We collect here some useful properties of the energy density~$W$. We refer to~\cite[Lemma~3.1]{A-N19} for more details.

\begin{lemma}\label{l.HMWw}
The function $W \colon \R \times \mathbb{M}^{2}_{s} \to [0,+\infty)$ is of class~$C^{1,1}_{\mathrm {loc}}$. Moreover, there exist two positive constants~$c$,~$C$ such that for every $z\in [0,1]$ and every $\mathbf{E}_{1}, \mathbf{E}_{2} \in \mathbb{M}^{2}_{s}$ the following holds:
\begin{itemize}
\item[$(a)$] $\big(\partial_{\mathbf{E}} W (z, \mathbf{E}_{1}) - \partial_{\strain} W (z, \mathbf{E}_{2}) \big) {\,:\,} ( \mathbf{E}_{1} - \mathbf{E}_{2} ) \geq c | \mathbf{E}_{1} - \mathbf{E}_{2} |^{2}$; \label{e.W1}
\item[$(b)$] $\big| \partial_{\mathbf{E}} W (z, \mathbf{E}_{1} ) - \partial_{\mathbf{E}} W (z, \mathbf{E}_{2} ) \big| \leq C |\mathbf{E}_{1} - \mathbf{E}_{2} |$; \label{e.W2} 
\item[$(c)$] $| \partial_{\mathbf{E}} W ( z , \mathbf{E} ) | \leq C | \mathbf{E} | $. \label{e.W3}
\end{itemize}
\end{lemma}

Let~$\Om$ be an open bounded subset of~$\R^{2}$ with Lipschitz boundary~$\partial\Om$. For later use, we also fix~$\partial_{D}\Om\subseteq \partial\Om$ regular in the sense of Gr\"oger~\cite{MR990595}. For every $u\in H^{1}(\Om;\R^{2})$ and $z\in H^{1}(\Om) \cap L^{\infty}(\Om)$ we define the \emph{elastic energy}
\begin{equation}\label{elastic}
 \E (u,z) \coloneq 
\tfrac{1}{2} \int_{\Om} W(z, \strain(u)) \,\di x\, ,
\end{equation}
where~$\strain(u) = \tfrac12 ( \nabla{u}+(\nabla{u})^{T} )$ denotes the strain.

We introduce the \emph{dissipation potential} associated to the phase field variable $z\in H^{1}(\Om) \cap L^{\infty}(\Om)$ given by
\begin{equation}\label{dissipation}
\D(z):=\tfrac{1}{2}\int_{\Om} |\nabla{z}|^{2} + f(z) \,\di x\,.
\end{equation}
Here, we assume the degradation function $f\colon \R\to [0,+\infty)$ to be of class~$C^{1,1}_{loc}$, strongly convex, and such that $0\leq f(1) \leq f(z)$. The prototypical example is $f(z) = (z-1)^{2}$. However, many different degradation functions have been extensively considered in the fracture mechanics literature (see, e.g.,~\cite{MR3304294, PhysRevLett.87.045501, Wu_JMPS17}).

The \emph{total energy} $\F \colon H^{1}(\Om;\R^{2}) \times  (H^{1}(\Om) \cap L^{\infty}(\Om)) \to [0,+\infty)$ of the system is given by the sum of elastic energy~\eqref{elastic} and dissipation potential~\eqref{dissipation}
\begin{equation}\label{totalenergy}
\F(u,z)\coloneq \E(u,z) + \D (z)\,.
\end{equation}
Notice that, in comparison with~\eqref{intro7}, we have fixed $\varepsilon = \tfrac{1}{2}$.


An important role in the definition of evolution we consider in this work is played by the following notion of \emph{unilateral $L^{2}$-slope} (see also~\cite[Definition~1.1]{MR3945577}).

\begin{definition}\label{definitionSlope}
	For $u\in H^{1}(\Om;\R^{2})$ and $z \in H^{1}(\Om) \cap L^{\infty}(\Om)$ we define the  \emph{unilateral $L^2$-slope} of~$\F$ with respect to $z$ at the point $(u,z)$ as
	\begin{equation}\label{Slope2}
		\abs{ \partial_z^{-}\F } (u,z) \coloneq \limsup_{\substack{ v \to z \\ v \in H^{1}(\Om) \cap L^{\infty}(\Om) , \, v \leq z }} \frac{[\F(u,z) - \F(u,v)]_+}{\norm{z-v}_{L^2}}\,,
	\end{equation}
where the convergence is intended in the $L^{2}$-topology.
\end{definition}

\begin{remark}
The minus sign appearing in the notation $|\partial^-_z \F |$ reminds that only negative variations are allowed and it should not be confused with a similar notation for the relaxed slope (see, e.g.,~\cite[Section~2.3]{MR2401600}).
\end{remark}

For $u\in H^{1}(\Om;\R^{2})$ and $z, \varphi \in H^{1}(\Om) \cap L^{\infty}(\Om)$ there exists finite the partial derivative of~$\F$ with respect to~$z$, i.e.,
\begin{equation}\label{derF}
\begin{split}
\partial_{z}\F(u,z)[\varphi]& = \int_{\Om} \partial_{z} W(z, \strain (u)) \varphi \,\di x+\int_{\Om}\nabla{z}{\,\cdot\,}\nabla{\varphi} + f'(z) \varphi\,\di x
\\
&
= \int_{\Om} h'(z)\varphi (\mu |\strain_{d}(u)|^{2} + \kappa |\strain_{v}^{+}(u)|^{2})\,\di x +\int_{\Om}\nabla{z}{\,\cdot\,}\nabla{\varphi} + f'(z) \varphi\,\di x\,.
\end{split}
\end{equation}
The natural relationship between partial derivatives~\eqref{derF} and slope~\eqref{Slope2} is stated in the next lemma, whose proof can be found, for instance, in~\cite[Lemma~2.3]{MR2401600} or~\cite[Lemma~2.2]{Negri_ACV}.

\begin{lemma}\label{SlopeLemma}
	For $u\in H^{1}(\Om;\R^{2})$ and $z \in H^{1}(\Om) \cap L^{\infty}(\Om)$ there holds
	\begin{displaymath}
		\abs{ \partial_z^{-}\F} (u,z) = \sup \,\{ -\partial_{z} \F(u,z)[\varphi] : \varphi \in H^{1}(\Om) \cap L^{\infty}(\Om), \, \varphi \leq 0 , \, \norm{\varphi}_{L^2} \leq 1\} \,.
	\end{displaymath}
\end{lemma}


Finally, let us define, for $u \in H^{1}(\Om;\R^{2})$ and $z\in H^{1}(\Om) \cap L^{\infty}(\Om)$, the functional 
\begin{displaymath}
	\mathcal{P} ( u , z, w ) \coloneq \int_\Om \partial_{\strain} W(z, \strain (u)) [\strain(w)] \,\di x = \partial_{u} \F( u, z) [w] \,.
\end{displaymath}


%


We are now in a position to give the precise definition of viscous and vanishing viscosity evolutions we consider in this paper.


\begin{definition}\label{d.Evolution}
Let $\delta>0$, $T>0$, and $g \in H^1 ([0,T]; W^{1,p} (\Om; \R^2))$ for some $p>2$.
Let $u_0 \in H^{1}(\Om;\R^{2})$ with $u_0 = g(0)$ on~$\partial_D \Omega$ and let $z_0 \in H^1(\Omega ; [0,1])$  be such that
\begin{eqnarray}
&&\displaystyle   u_0 \in \argmin\, \{ \E (u ,z_{0}) : \text{$u\in H^{1}(\Om;\R^{2})$ with $u=g(0)$ on~$\partial_D \Om$} \}\,,\label{equilibriumu0} \\[1mm]
&& \displaystyle z_0 \in \argmin \{\F(u_0, z) : \text{$z\in H^{1}(\Om)$ and $z\leq z_0$}\} \,.\label{equilibriumz0}
\end{eqnarray}

We say that a pair $(u_\delta , z_\delta ) \colon [0,T] \to H^{1}(\Om;\R^{2}) \times H^{1}(\Om)$ is a \emph{viscous evolution} for the energy~$\F$ with initial condition~$(u_0, z_0)$ and boundary condition~$g$ if the following properties are satisfied:

\smallskip
\begin{itemize}\setlength\itemsep{3pt}
\item[$(a)$] \emph{Time regularity}: $u_\delta \in C ([0,T]; H^{1}(\Om;\R^{2}))$ and $z_\delta \in H^{1}( [0,T] ;H^{1} (\Om))$ with $u_{\delta} (0) = u_{0}$ and
\phantom{-----------------------} $z_{\delta} (0) = z_0$;

\item[$(b)$] \emph{Irreversibility}: $t\mapsto z_{\delta}(t)$ is non-increasing (i.e., $z_{\delta}(\tau)\leq z_{\delta}(t)$ a.e.~in $\Om$ for every $0\leq t\leq \tau \leq T$) \phantom{-------------------}  and $0 \le z_{\delta}(t) \le 1$ for every $t \in [0,T]$;
\item[$(c)$] \emph{Displacement equilibrium}: for every $t\in[0,T]$ we have $u_{\delta} (t) = g(t)$ on~$\partial_D \Omega$ and 
\begin{displaymath}
u_{\delta} (t)  \in \argmin\, \{ \E (u ,z_{\delta}(t)) : \text{$u\in H^{1}(\Om;\R^{2})$ with $u=g(t)$ on~$\partial_D \Om$} \}\,;
\end{displaymath}
\item[$(d)$] \emph{Energy balance}: the map $t\mapsto \F(u_{\delta}(t),z_{\delta}(t))$ is absolutely continuous and for a.e.~$t\in[0,T]$ it holds
\begin{equation}\label{e.enbaldelta}
\dot \F(u_{\delta}(t),z_{\delta}(t))=-\tfrac{\delta}{2} \|\dot{z}_{\delta}(t)\|_{L^2}^{2} - \tfrac{1}{2 \delta} |\partial_z^{-}\F|^{2}(u_{\delta}(t) , z_{\delta}(t)) + \mathcal{P} ( u_{\delta}(t) , z_{\delta}(t), \dot g (t))\,.
\end{equation}
\end{itemize}
\end{definition}


Our first goal is to prove the convergence to a unilateral $L^{2}$-gradient flow of the \emph{time discrete} solutions obtained by an alternating minimization scheme (see~\eqref{minu}-\eqref{minz} for the algorithm and Theorem~\ref{t.viscous} for the convergence result). Our second aim is to characterize the limit $\delta \to 0$, of the above evolution. The limit solutions are described in the following definition.

\begin{definition}\label{d.vanevolution}
Let $T>0$ and $g \in H^1 ([0,T]; W^{1,p} (\Om; \R^2))$ for some $p>2$.
Let $u_0 \in H^{1}(\Om;\R^{2})$ with $u_0 = g(0)$ on~$\partial_D \Omega$ and $z_0 \in H^1(\Omega ; [0,1])$  be such that~\eqref{equilibriumu0} and~\eqref{equilibriumz0} are satisfied.

For $S\in (0,+\infty)$ we say that a triple $(t, u , z ) \colon [0,S] \to [0,T]\times H^{1}(\Om;\R^{2}) \times H^{1}(\Om)$ is a \emph{vanishing viscosity evolution} for the energy~$\F$ with initial condition~$(u_0, z_0)$ and boundary condition~$g$ if the following properties are satisfied:

\smallskip
\begin{itemize}\setlength\itemsep{3pt}
\item[$(a)$] \emph{Time regularity}: $t\in W^{1,\infty}(0,S)$, $u \in C ([0,S]; H^{1}(\Om;\R^{2}))$, and $z_\delta \in W^{1,\infty}( [0,S] ;H^{1} (\Om))$ with $t(0) = 0$, \phantom{-----------------------}$t(S) = T$, $u (0) = u_{0}$, and
$z (0) = z_0$;

\item[$(b)$] \emph{Irreversibility}: $s \mapsto z (s)$ is non-increasing and $0 \le z (s) \le 1$ for every $s \in [0,S]$;

\item[$(c)$] \emph{Normalization}: $t'(s) + \| z'(s) \|_{H^{1}} \leq 1$ for a.e.~$s\in[0,S]$;
\item[$(d)$] \emph{Displacement equilibrium}: for every $s \in [0,S]$ we have $u (s) = g(t(s))$ on~$\partial_D \Omega$ and 
\begin{displaymath}
u (s)  \in \argmin\, \{ \E (u ,z(s)) : \text{$u\in H^{1}(\Om;\R^{2})$ with $u=g(t(s))$ on~$\partial_D \Om$} \}\,;
\end{displaymath}
\item[$(e)$] \emph{Energy balance}: the map $s \mapsto \F(u (s),z (s))$ is absolutely continuous and for a.e.~$s \in [0,S]$ it holds
\begin{equation}\label{e.viscousenergybalance}
\F'(u (s),z (s) ) = - |\partial_z^{-}\F| ( u (s) , z (s) ) \| z'(s) \|_{2} + \mathcal{P} ( u (s) , z(s), \dot g (t(s))) t'(s) \,.
\end{equation}
\end{itemize}
\end{definition}

\medskip
Next lemma provides a regularity property needed in our setting. The complete proof can be found in~\cite[Proposition~3.6]{A-N19}. For a more general statement we refer to~\cite{HerzogMeyerWachsmuth_JMAA11}.

\begin{lemma} \label{l.regLp}
 Let $g \in H^1{} ( [0,T] ; W^{1, p} (\Omega; \R^2) )$ for $p>2$. For $t \in [0,T]$ and $z \in H^1(\Omega; [0,1])$ let us set
\begin{displaymath}
 u (t,z)  \coloneq \argmin \,\{ \E (u ,z ) : \text{$u\in H^{1}(\Om;\R^{2})$ with $z=g(t)$ on~$\partial_{D} \Om$} \}\,.  
 \end{displaymath}
Then there exist an exponent $2 < r < p$ and a constant $C>0$ such that for every $t_1, t_2 \in [0,T]$ and every $z_1,z_2 \in H^1(\Omega; [0,1])$ it holds
\begin{displaymath}
	\| u  ( t_2 , z_2) - u  ( t_1 , z_1 ) \|_{W^{1,r}} \le C ( \| g(t_2) - g(t_1) \|_{W^{1,r}} + \| g \|_{L^\infty ( 0,T; \,W^{1,p})} \, \| z_2 - z_1 \|_{\ell}  ) , 
\end{displaymath}
where $1/\ell = 1/ r - 1/ p$.
\end{lemma} 

We collect here some technical results that will be useful in the next sections.

\begin{lemma}\label{lsc-F}
 Let $u_m, u\in H^{1}(\Om;\R^{2})$ 
and $z_m, z\in H^{1}(\Om;[0,1])$.  If $u_m \weakto u$ in~$H^{1}(\Om;\R^{2})$ and $z_m \rightharpoonup z$  in~$H^{1}(\Om)$, then
$$ 
   \F(u,z)\leq\liminf_{m\to\infty} \F (u_m,z_m) \,.
$$ 
\end{lemma}

\begin{proof} 
The lower semicontinuity of $\D$ is obvious by convexity. The semicontinuity of $\E$ follows for instance from~\cite[Theorem 7.5]{Fonseca2007}. 
\end{proof}

Now we state a semicontinuity property of the slope~$|\partial_z^{-}\F|$. The proof can be found in~\cite[Lemma~3.9]{A-N19}.

\begin{lemma}\label{prop1}
Let $u_m, u\in H^{1}(\Om;\R^{2})$ 
and $z_m, z \in H^{1}(\Om)\cap L^{\infty}(\Om)$. 
If $u_m\to u$ in~$\UU$ and $z_m\rightharpoonup z$ weakly  in~$H^{1}(\Om)$, then
\begin{displaymath}
|\partial_z^{-}\F|(u,z)\leq\liminf_{m\to\infty} |\partial_z^{-}\F|(u_m,z_m)\,.
\end{displaymath}
\end{lemma}


Finally, the next lemma shows a continuity property of minimizer of~$\E(\cdot, z)$ w.r.t.~$z \in H^{1}(\Om;[0,1])$ and $t\in [0,T]$. We refer to~\cite[Proposition~3.7]{A-N19} for the complete proof.

\begin{lemma}\label{lemma1}
Let $g_m,g_{\infty},u_m,u_{\infty}\in H^{1}(\Om;\R^{2})$ and let $z_m,z_{\infty}\in H^{1}(\Om;[0,1])$. Assume that $g_m \to g_{\infty}$ in~$H^{1}(\Om;\R^{2})$, $u_m \rightharpoonup u_{\infty}$ weakly in~$H^{1}(\Om;\R^{2})$, $z_m\rightharpoonup z_{\infty}$ weakly in~$H^{1}(\Om)$, and that
\begin{displaymath}
u_m \in \argmin\, \{ \E (u ,z_m ) : \text{ $u\in H^{1}(\Om;\R^{2})$ with $u=g_m$ on~$\partial_D \Om$}\} \,. 
\end{displaymath}
Then,
\begin{displaymath}
u_\infty \in \argmin \,\{ \E (u ,z_\infty ) : \text{ $u\in H^{1}(\Om;\R^{2})$ with $u=g_\infty$ on~$\partial_D \Om$} \}
\end{displaymath}
and $u_m\to u_{\infty}$ strongly in $H^{1}(\Om;\R^{2})$.
\end{lemma}

\section{Alternate minimization algorithm and a priori estimates}\label{s.algorithm}

Let us fix a time horizon $T>0$ and a Dirichlet boundary datum $g \in H^{1} ([0,T] ; W^{1,p}(\Om;\R^{2}))$, with~ $p\in(2, +\infty)$, to be applied on the Dirichlet part of the boundary~$\partial_D \Om \subseteq \partial \Om$. For fixed $\delta>0$, the construction of a gradient flow of~$\F$ is done by time-discretization (see, e.g.,~\cite{MR2401600}). Here we couple this standard procedure with an alternate (or staggered) minimization algorithm, which is by now typical in computational fracture mechanics~\cite{MR3376787, MR2341850, MR1745759, Burke2010, Burke2013}. 

Let us describe the minimization scheme. For every $k\in\mathbb{N}$ we define the time step $\tau_k\coloneq T/k$ and the time nodes $t_{i}^{k}\coloneq i \tau_{k}$ for $i \in \{ 0,\ldots, k \}$. For $i=1, \ldots, k$ assume that we know the configuration~$(u^{k}_{i-1} , z^{k}_{i-1})$ at time~$t^{k}_{i-1}$. The new state~$(u^{k}_{i} , z^{k}_{i})$ at time~$t^{k}_{i}$ is constructed as limit of an alternate minimization procedure: for $j=0$ we set $u^{k}_{i,0}\coloneq u^{k}_{i-1}$ and $z^{k}_{i,0}\coloneq z^{k}_{i-1}$, while for $j\geq 1$ we define
\begin{eqnarray}
&&\displaystyle u^{k}_{i,j} \coloneq \argmin\, \{ \F (u, z^{k}_{i,j-1}) : \, u\in H^{1}(\Om;\R^{2}),\, \text{$u = g(t^{k}_{i})$ on $\partial_{D} \Om$}\}\,, \label{minu} \\[1mm]
&& \displaystyle z^{k}_{i,j} \coloneq \argmin \, \Big\{ \F(u^{k}_{i,j} , z) + \tfrac{\delta}{2\tau_k} \| z - z^{k}_{i-1} \|_{2}^{2} : \, z\in H^{1}(\Om) , \, z\leq z^{k}_{i-1}\Big\}\,. \label{minz}
\end{eqnarray}
We notice that the solutions of~\eqref{minu} and~\eqref{minz} exist and are unique by strict convexity.

\begin{remark}
We stress that the irreversibility constraint, stating that the damage (or the crack) can only grow and healing is not allowed, is expressed in~\eqref{minz} by the inequality~$z\leq z^{k}_{i-1}$, which only involves the configuration at time~$t^{k}_{i-1}$. This is in contrast with the recent theoretical literature developed in~\cite{Almi2017, A-N19, MR3606956, Knees2018, MR3669838}, where the irreversibility has been implemented in a stronger way by imposing $z \leq z^{k}_{i,j-1}$. The latter choice, even if mathematically correct, could lead to the accumulation of numerical error in the simulation of the fracture process. For this reason, in most of the applications in computational fracture mechanics (see, e.g.,~\cite{MR3376787, MR2341850, MR1745759, Burke2010, Burke2013}) the constraint in~\eqref{minz} is considered.

Again in contrast with~\cite{Almi2017, A-N19, MR3606956, MR3669838}, in~\eqref{minz} we have perturbed the functional~$\F$ with an $L^{2}$-penalization of the distance from~$z^{k}_{i-1}$. This is necessary in order to carry out our analysis, and leads in the limit as~$k\to\infty$ to the construction of a gradient flow of the energy~$\F$ according to Definition~\ref{d.Evolution}. A quasi-static evolution is recovered only in the limit as the viscosity parameter~$\delta$ tends to~$0$. We refer to Section~\ref{s.evolution} for the full discussion. We further notice that a similar penalization has been used also in~\cite{MR3945577, MR3021776, MR3332887,MR3893258, Negri_ACV}. However, only~\cite{MR3945577} deals with an alternate minimization procedure, where the minimization in~\eqref{minz} is unconstrained and the irreversibility is imposed a posteriori by a simple pointwise minimization, which made the study of the viscosity limit $\delta \to 0$ not accessible.
\end{remark}

In the following proposition we prove the convergence, up to subsequence, of the sequence $(u^{k}_{i,j} , z^{k}_{i,j})$.

\begin{proposition}\label{p.convj}
Fix $\delta>0$, $k\in\mathbb{N}$, and $i\in\{1,\ldots, k\}$, and let $u^{k}_{i,j} \in H^{1}(\Om;\R^{2})$ and $z^{k}_{i,j}\in H^{1}(\Om)$ be defined by~\eqref{minu}-\eqref{minz}. Then, there exist $\bar{u} \in H^{1}(\Om;\R^{2})$ and $\bar{z} \in H^{1}(\Om;[0,1])$ such that, up to a subsequence, $u^{k}_{i,j} \to \bar{u}$ in~$H^{1}(\Om;\R^{2})$ and $z^{k}_{i,j}\rightharpoonup \bar z$ weakly in~$H^{1}(\Om)$ as $j\to\infty$, $\bar{u} = g(t^{k}_{i})$ on~$\partial_D \Om$, and
\begin{eqnarray}
&& \displaystyle \bar{u} = \argmin \, \{\F(u, \bar z) : \, u\in H^{1}(\Om;\R^{2}), \, \text{$ u = g(t^{k}_{i})$ on $\partial_{D}\Om$}\}\,, \label{e.1}\\[1mm]
&& \displaystyle \bar{z} = \argmin \, \Big\{ \F(\bar{u}, z) + \tfrac{\delta}{2\tau_{k}} \| z - z^{k}_{i-1} \|_{2}^{2} : \, z\in H^{1}(\Om),\, z\leq z^{k}_{i-1} \Big\}\,. \label{e.2}
\end{eqnarray}
\end{proposition}

\begin{proof}
By definition of~$u^{k}_{i,j}$ and of~$z^{k}_{i,j}$ we have that
\begin{displaymath}
\begin{split}
\F ( u^{k}_{i,j} , v^{k}_{i,j}) +& \tfrac{\delta}{2 \tau_{k}} \| z^{k}_{i,j} - z^{k}_{i-1}\|_{2}^{2}  \leq \F ( u^{k}_{i,j} , z^{k}_{i,j-1} ) + \tfrac{\delta}{2 \tau_{k}} \| z^{k}_{i,j-1} - z^{k}_{i-1} \|_{2}^{2} 
\\
&
\leq \F ( u^{k}_{i,j-1} , z^{k}_{i,j-1}) + \tfrac{\delta}{2 \tau_{k}} \| z^{k}_{i,j-1} - z^{k}_{i-1} \|_{2}^{2}
\leq \ldots \leq \F ( u^{k}_{i-1} + g ( t^{k}_{i} ) - g ( t^{k}_{i-1} ), z^{k}_{i-1} ) \,.
\end{split}
\end{displaymath}
Hence, the sequences $u^{k}_{i,j}$ and~$z^{k}_{i,j}$ are bounded in~$H^{1}$, uniformly w.r.t.~$j$. Thus, there exist a subsequence~$j_{l}$,~$\bar{u} \in H^{1}(\Om;\R^{2})$, and $\bar{z}\in H^{1}(\Om ; [0,1])$ such that $u^{k}_{i, j_{l}} \rightharpoonup \bar{u}$ weakly in~$H^{1}(\Om;\R^{2})$ and $z^{k}_{i,j_{l}} \rightharpoonup \bar{z}$ weakly in~$H^{1}(\Om)$ as $l\to\infty$. Up to a further (not relabeled) subsequence, we may assume that the above convergences are strong in~$L^{2}(\Om)$ and that $z^{k}_{i,j_{l}-1}\rightharpoonup \hat{z}$ weakly in~$H^{1}(\Om)$, for some $\hat{z} \in H^{1}(\Om;[0,1])$.

By Lemma~\ref{lemma1} and by the above convergences, we immediately deduce that $u^{k}_{i, j_{l}}\to \bar{u}$ in~$H^{1}(\Om;\R^{2})$ and that
\begin{displaymath}
\bar{u} = \argmin\, \{\F(u, \hat{z}) : \, \text{$u\in H^{1}(\Om;\R^{2})$, $u= g(t^{k}_{i})$ on~$\partial_{D}\Om$}\}\,.
\end{displaymath}
Moreover, we have that~\eqref{e.2} is satisfied. Indeed, by~\eqref{minz} we have that for every~$z\in H^{1}(\Om)$ with $z\leq z^{k}_{i-1}$
\begin{displaymath}
\F(u^{k}_{i,j_{l}}, z^{k}_{i,j_{l}}) + \tfrac{\delta}{2\tau_{k}} \| z^{k}_{i,j_{l}} - z^{k}_{i-1} \|_{2}^{2} \leq \F( u^{k}_{i,j_{l}}, z) +\tfrac{\delta}{2\tau_{k}} \| z - z^{k}_{i-1} \|_{2}^{2}\,.
\end{displaymath}
Therefore, passing to the liminf as $l\to\infty$ in the previous inequality, applying Lemma~\ref{lsc-F}, and exploiting the strong convergence of~$u^{k}_{i,j_{l}}$ we get~\eqref{e.2}.

It remains to prove that $\hat{z} = \bar{z}$. To do this, we notice that
\begin{displaymath}
\begin{split}
\F(u^{k}_{i,j_{l}}, z^{k}_{i, j_{l}-1}) +  \tfrac{\delta}{2\tau_{k}} \| z^{k}_{i,j_{l}-1} - z^{k}_{i-1} \|_{2}^{2} & \leq \F(u^{k}_{i,j_{l}-1}, z^{k}_{i, j_{l}-1}) + \tfrac{\delta}{2\tau_{k}} \| z^{k}_{i,j_{l}-1} - z^{k}_{i-1} \|_{2}^{2}
\\
&
\leq \ldots \leq \F(u^{k}_{i,j_{l-1}}, z^{k}_{i, j_{l-1}}) + \tfrac{\delta}{2\tau_{k}} \| z^{k}_{i,j_{l-1}} - z^{k}_{i-1} \|_{2}^{2}
\\
&
\leq \F(u^{k}_{i,j_{l-1}}, \bar{z}) + \tfrac{\delta}{2\tau_{k}} \| \bar{z} - z^{k}_{i-1} \|_{2}^{2}\,.
\end{split}
\end{displaymath}
Hence, passing to the liminf in the previous chain of inequalities and applying again Lemma~\ref{lsc-F} we obtain
\begin{displaymath}
\F(\bar{u}, \hat{z}) + \tfrac{\delta}{2\tau_{k}} \| \hat z - z^{k}_{i-1} \|_{2}^{2} \leq \F(\bar u, \bar{z}) +  \tfrac{\delta}{2\tau_{k}} \| \bar{z} - z^{k}_{i-1} \|_{2}^{2}\,.
\end{displaymath}
By uniqueness of minimizer, this implies that $\hat{z} = \bar{z}$, and the proof is thus concluded.
\end{proof}

In view of Proposition~\ref{p.convj}, we are allowed to define
\begin{displaymath}
u^{k}_{i}\coloneq \lim_{j\to\infty} u^{k}_{i,j}\qquad\text{and}\qquad z^{k}_{i}\coloneq \lim_{j\to\infty} z^{k}_{i,j}\,,
\end{displaymath}
where the limits are intended to be up to a subsequence and strong in~$H^{1}(\Om;\R^{2})$ and weak in~$H^{1}(\Om)$, respectively. In particular,~$u^{k}_{i}$ and~$z^{k}_{i}$ solve
\begin{eqnarray}
&& \displaystyle u^{k}_{i} = \argmin \, \{\F(u,  z^{k}_{i} ) : \, u\in H^{1}(\Om;\R^{2}), \, \text{$ u = g(t^{k}_{i})$ on $\partial_{D}\Om$}\}\,, \label{e.3}\\[1mm]
&& \displaystyle z^{k}_{i} = \argmin \, \Big\{ \F(u^{k}_{i}, z) + \tfrac{\delta}{2\tau_{k}} \| z - z^{k}_{i-1} \|_{2}^{2} : \, z\in H^{1}(\Om),\, z\leq z^{k}_{i-1} \Big\}\,. \label{e.4}
\end{eqnarray}

In the following two propositions we prove the boundedness of $u^{k}_{i}$ and~$z^{k}_{i}$ together with a discrete energy inequality.

\begin{proposition}\label{p.2}
For every $k\in\mathbb{N}$ and every $i\in\{1,\ldots,k\}$ it holds
\begin{eqnarray}
&\displaystyle |\partial^-_z\F|(u^{k}_{i}, z^{k}_{i}) =  \tfrac{\delta}{\tau_k} \| z^{k}_{i} - z^{k}_{i-1 }\|_{2}  \,,\label{e.5}\\
&\displaystyle \partial_{z} \F ( u^{k}_{i} , z^{k}_{i} ) [ z^{k}_{i} - z^{k}_{i-1} ] = - | \partial^-_z \F | ( u^{k}_{i} , z^{k}_{i} ) \| z^{k}_{i} - z^{k}_{i-1} \|_{2}\,.\label{e.6}
\end{eqnarray}
\end{proposition}

\begin{proof}
By~\eqref{e.4} we have that
\begin{displaymath}
\partial_{z} \F (u^{k}_{i} , z^{k}_{i} ) [ \varphi ] + \tfrac{\delta}{\tau_{k}} \int_{\Om} ( z^{k}_{i} - z^{k}_{i-1}) \varphi \, \di x = 0\qquad\text{for every $\varphi\in H^{1}(\Om)\cap L^{\infty}(\Om)$ with $\varphi\leq 0$}\,.
\end{displaymath}
Then, by Lemma~\ref{SlopeLemma} and by the density of~$H^{1}(\Om)\cap L^{\infty}(\Om)$ in~$L^2(\Om)$, we get that
\begin{displaymath}
\begin{split}
|\partial^-_z \F | ( u^{k}_{i} , z^{k}_{i} ) & = \vphantom{\int} \sup \, \{ - \partial_{z} \F ( u^{k}_{i} , z^{k}_{i}) [ \varphi ] : \,\varphi\in H^{1}(\Om)\cap L^{\infty}(\Om) , \, \varphi \leq 0 , \, \| \varphi \|_{2} \leq 1 \} 
\\
&
= \max \, \Big \{ \tfrac{\delta}{\tau_{k}} \int_{\Om} (z^{k}_{i} - z^{k}_{i-1} ) \varphi \, \di x : \, \varphi \in L^{2}(\Om),\,\varphi \leq 0, \, \| \varphi \|_{2} \leq 1 \Big \}
\\
&
= \tfrac{\delta}{\tau_{k}} \int_{\Om} (z^{k}_{i} - z^{k}_{i-1}) \frac{( z^{k}_{i} - z^{k}_{i-1})}{ \| z^{k}_{i} - z^{k}_{i-1} \|_{2}} \, \di x
= \tfrac{\delta}{\tau_{k}} \| z^{k}_{i} - z^{k}_{i-1} \|_{2} \,,
\end{split}
\end{displaymath}
which yields~\eqref{e.5} and~\eqref{e.6}.
\end{proof}

We now define the following interpolation functions:
\begin{eqnarray}
&&\displaystyle z^{\delta}_{k}(t)\coloneq z^{k}_{i} + \frac{z^{k}_{i+1} - z^{k}_{i}}{\tau_{k}} ( t - t^{k}_{i} )\qquad\text{for every $t \in [ t^{k}_{i} , t^{k}_{i+1} )$}\,,\label{int1}\\
&&\displaystyle \bar{u}^{\delta}_{k} (t) \coloneq u^{k}_{i} \,, \quad \bar{z}^{\delta}_{k} (t) \coloneq z^{k}_{i} \,,\quad t_{k} (t) \coloneq t^{k}_{i} \qquad\text{for every $t \in ( t^{k}_{i-1} , t^{k}_{i} ]$ } \,, \label{int2} \\ [1mm]
&&\displaystyle \ubar{u}^{\delta}_{k} (t) \coloneq u^{k}_{i} \,, \quad \ubar{z}^{\delta}_{k} (t) \coloneq v^{k}_{i} \,,\qquad\text{for every $t \in [ t^{k}_{i} , t^{k}_{i+1} )$ } \,. \label{int3}
\end{eqnarray}
 
We notice that at this point, arguing as in~\cite{MR3945577}, we could already show the convergence of the sequence $(\bar u^{\delta}_{k}, z^{\delta}_{k})$ to a viscous evolution~$(u_{\delta}, z_{\delta})$, without passing through a priori estimates showing time regularity of~$z^{\delta}_{k}$. The most delicate point would indeed be the energy balance~(d) of Definition~\ref{d.Evolution}, which would anyway follow from convexity and Riemann sum arguments as in~\cite{MR3945577}. However, such an analysis would preclude the study of the limit~$\delta\to 0$.

Therefore, in the next two lemmas we provide some a priori estimates for the discrete evolutions defined in~\eqref{int1}-\eqref{int3}. Precisely, we show in Lemma~\ref{l.H1H1} that~$z^{\delta}_{k}$ is bounded in $H^{1}([0,T]; H^{1}(\Om))$ w.r.t.~$k$. Since the bound is not uniform in~$\delta>0$, we further prove in Lemma~\ref{l.length} that the length of the curve $t\mapsto z^{\delta}_{k}(t)$, namely 
\begin{displaymath}
\int_{0}^{T} \| \dot{z}^{\delta}_{k} (\tau)\|_{H^{1}} \, \di \tau
\end{displaymath}
is bounded uniformly w.r.t.~$k$ and~$\delta$. This will allow us to consider the limit~$\delta\to 0$ in Theorem~\ref{t.vanevolution}.


\begin{lemma}\label{l.H1H1}
There exists a positive constant $C$ independent of $\delta$ and~$k$ such that for every $t\in [0,T]$
\begin{eqnarray}
&&\displaystyle \delta \| \dot{z}^{\delta}_{k} (t) \|_{2} \leq C e^{\frac{C}{\delta} t_{k}(t)} \,, \label{e.H11} \\ [2mm]
&&\displaystyle \delta \int_{0}^{t_{k}(t)} \| \dot{z}_{k}^{\delta} (\tau) \|_{H^{1}}^{2} \, \di \tau \leq C e^{\frac{C}{\delta}t_{k}(t)}\,. \label{e.H12}
\end{eqnarray}
\end{lemma}

\begin{proof}
We follow here the lines of~\cite[Propositions~4.6 and~5.7]{MR3021776} and~\cite[Proposition~2.8]{MR3454016}. For $\delta$ and $k$ fixed, for simplicity of notation we set $\dot{z}_{i} \coloneq \tfrac{z^{k}_{i} - z^{k}_{i-1}}{\tau_{k}}$ and $\dot{g}_{i} \coloneq \tfrac{ g(t^{k}_{i}) - g(t^{k}_{i-1}) }{\tau_{k}}$, $i=1, \ldots, k$.

In view of~\eqref{e.4} we have that
\begin{eqnarray}
&& \displaystyle \partial_{z} \F(u^{k}_{i}, z^{k}_{i})[\varphi] + \tfrac{\delta}{\tau_{k}} \int_{\Om} (z^{k}_{i} - z^{k}_{i-1}) \varphi\, \di x \geq 0\qquad\text{for every $\varphi\in H^{1}(\Om)$, $\varphi \leq 0$}\,.\label{e.11} \\[2mm]
&& \displaystyle \partial_{z} \F(u^{k}_{i}, z^{k}_{i})[\dot{z}_{i}] + \delta \|\dot{z}_{i}\|_{2}^{2} = 0\,. \label{e.12}
\end{eqnarray}
Let us consider $i\geq 2$. Inserting~$\dot{z}_{i}$ as a test function in~\eqref{e.11} at time~$t^{k}_{i-1}$ we have that
\begin{displaymath}
\partial_{z} \F ( u^{k}_{i-1}, z^{k}_{i-1} )  [ \dot{z}_{i} ] + \delta \int_{\Om} \dot{z}_{i-1} \dot{z}_{i}\, \di x \geq 0\,.
\end{displaymath}
Subtracting~\eqref{e.12} from the previous inequality we get
\begin{equation}\label{e.13}
 \partial_{z} \F ( u^{k}_{i-1}, z^{k}_{i-1} )  [ \dot{z}_{i} ] - \partial_{z} \F(u^{k}_{i}, z^{k}_{i})[\dot{z}_{i}] \geq \delta \int_{\Om} ( \dot{z}_{i} - \dot{z}_{i-1} ) \dot{z}_{i}\, \di x \geq \tfrac{\delta}{2}( \| \dot{z}_{i}\|_{2}^{2} - \| \dot{z}_{i-1} \|_{2}^{2} ) \,,
\end{equation}
where, in the last step, we have used the inequality $2a(a-b)\geq a^{2} - b^{2}$.

We now estimate the left-hand side of~\eqref{e.13}. First, we split the difference using the definition~\eqref{totalenergy} of~$\F$, so that
\begin{equation}\label{e.14}
\begin{split}
 \partial_{z} \F ( u^{k}_{i-1}, z^{k}_{i-1} )  [ \dot{z}_{i} ] - \partial_{z} \F(u^{k}_{i}, z^{k}_{i})[\dot{z}_{i}] = & \ \partial_{z} \E(u^{k}_{i-1}, z^{k}_{i-1}) [\dot{z}_{i}]  - \partial_{z} \E(u^{k}_{i}, z^{k}_{i}) [\dot{z}_{i}] 
  \\
 &
 + \partial_{z} \D (z^{k}_{i-1})[\dot{z}_{i}] - \partial_{z}\D( z^{k}_{i}) [\dot{z}_{i}]\,.
\end{split}
\end{equation}
Since $f$ is strongly convex, there exists~$c \in (0, 1)$ such that
\begin{equation}\label{e.15}
\partial_{z} \D (z^{k}_{i-1})[\dot{z}_{i}] - \partial_{z}\D( z^{k}_{i}) [\dot{z}_{i}] = \int_{\Om} \nabla{( z^{k}_{i-1} - z^{k}_{i})} \nabla{\dot{z}_{i}} \, \di x + \int_{\Om} (f'(z^{k}_{i-1}) -f'( z^{k}_{i}) ) \dot{z}_{i}\, \di x \leq - c \tau_{k} \|\dot{z}_{i} \|^{2}_{H^{1}}\,.
\end{equation}
As for the first two terms on the right-hand side of~\eqref{e.14}, we write
\begin{equation}\label{e.16}
\begin{split}
\partial_{z} \E & (u^{k}_{i-1}, z^{k}_{i-1})  [\dot{z}_{i}]  - \partial_{z} \E(u^{k}_{i}, z^{k}_{i}) [\dot{z}_{i}] 
\\
&
= \partial_{z} \E(u^{k}_{i-1}, z^{k}_{i-1})  [\dot{z}_{i}]  - \partial_{z} \E(u^{k}_{i}, z^{k}_{i-1})  [\dot{z}_{i}]  +  \partial_{z} \E(u^{k}_{i}, z^{k}_{i-1})  [\dot{z}_{i}]  - \partial_{z} \E(u^{k}_{i}, z^{k}_{i}) [\dot{z}_{i}]
\\
&
= \int_{\Om} \dot{z}_{i} h'(z^{k}_{i-1}) \big(\mu|\strain_{d} (u^{k}_{i-1})|^{2} + \kappa | \strain_{v}^{+}(u^{k}_{i-1})|^{2} - \mu|\strain_{d} (u^{k}_{i})|^{2} - \kappa | \strain_{v}^{+}(u^{k}_{i})|^{2} \big) \,\di x 
\\
&
\qquad + \int_{\Om} \dot{z}_{i} ( h'(z^{k}_{i-1}) - h'(z^{k}_{i}) ) \big( \mu|\strain_{d} (u^{k}_{i})|^{2} + \kappa | \strain_{v}^{+}(u^{k}_{i})|^{2} \big)  \,\di x
 \\
&
\leq \int_{\Om} \dot{z}_{i} h'(z^{k}_{i-1}) \mu \big( |\strain_{d} (u^{k}_{i-1})| + |\strain_{d} (u^{k}_{i})| \big) \big( |\strain_{d} (u^{k}_{i-1})| - |\strain_{d} (u^{k}_{i})| \big) \,\di x  
\\
&
\qquad +\int_{\Om} \dot{z}_{i} h'(z^{k}_{i-1})\kappa \big( | \strain_{v}^{+}(u^{k}_{i-1})| +  | \strain_{v}^{+}(u^{k}_{i})| \big) \big( | \strain_{v}^{+}(u^{k}_{i-1})| -  | \strain_{v}^{+}(u^{k}_{i})| \big) \,\di x \,,
\end{split}
\end{equation}
where, in the last inequality, we have used the convexity of~$h$, which yields  $\dot{z}_{i} (h'(z^{k}_{i-1}) - h'(z^{k}_{i})) \leq 0$. Let $r,\ell \in(2,+\infty)$ be as in Lemma~\ref{l.regLp}. Since $0\leq z^{k}_{i} \leq 1$, by H\"older inequality we continue in~\eqref{e.16} with
\begin{equation}\label{e.17}
\begin{split}
\partial_{z} \E & (u^{k}_{i-1}, z^{k}_{i-1})  [\dot{z}_{i}]  - \partial_{z} \E(u^{k}_{i}, z^{k}_{i}) [\dot{z}_{i}] 
\leq \|\dot{z}_{i}\|_{\nu} \| u^{k}_{i} - u^{k}_{i-1} \|_{W^{1,r}} ( \| u^{k}_{i}\|_{W^{1,r}} + \| u^{k}_{i-1}\|_{W^{1,r}}) \,,
\end{split}
\end{equation}
for $\tfrac{1}{\nu} + \tfrac{2}{r} = 1$. By Lemma~\ref{l.regLp} we know that $\|u^{k}_{i}\|_{W^{1,r}}$ is bounded uniformly w.r.t.~$k$,~$i$, and~$\delta$. Furthermore, we deduce from~\eqref{e.17} that
\begin{equation}\label{e.18}
\begin{split}
\partial_{z} \E  (u^{k}_{i-1}, z^{k}_{i-1})  [\dot{z}_{i}]  - \partial_{z} \E(u^{k}_{i}, z^{k}_{i}) [\dot{z}_{i}] 
& \leq C \|\dot{z}_{i}\|_{\nu} \big(\| g(t^{k}_{i}) - g(t^{k}_{i-1}) \|_{W^{1,p}} + \| z^{k}_{i} - z^{k}_{i-1} \|_{\ell} \big)
\\
&
= C \tau_{k}  \|\dot{z}_{i}\|_{\nu}  \big(\| \dot{g}_{i} \|_{W^{1,p}} + \| \dot{z}_{i} \|_{\ell} \big)
\end{split}
\end{equation}
for some positive constant~$C$ independent of~$k$,~$i$, and~$\delta$. Setting~$\lambda\coloneq \max\{ \nu, \ell\}$ and applying Young inequality to~\eqref{e.18} we infer that
\begin{displaymath}
\partial_{z} \E  (u^{k}_{i-1}, z^{k}_{i-1})  [\dot{z}_{i}]  - \partial_{z} \E(u^{k}_{i}, z^{k}_{i}) [\dot{z}_{i}] 
\leq C_{1} \tau_{k}  \|\dot{z}_{i}\|_{\lambda}^{2} + C_{2} \tau_{k} \| \dot{g}_{i} \|^{2}_{W^{1,p}}\,.
\end{displaymath}
Finally, by interpolation and Sobolev inequality we obtain
\begin{equation}\label{e.19}
\begin{split}
\partial_{z} \E  (u^{k}_{i-1}, z^{k}_{i-1})  [\dot{z}_{i}]  - \partial_{z} \E(u^{k}_{i}, z^{k}_{i}) [\dot{z}_{i}] & \leq  \eps \tau_{k} \| \dot{z}_{i}\|_{H^{1}}^{2} + C_{\eps} \tau_{k} \|\dot{z}_{i}\|_{2}^{2} + C_2\tau_{k} \| \dot{g}_{i} \|^{2}_{W^{1,p}}
\end{split}
\end{equation}
for $\eps, C_{\eps} >0$. Choosing $\eps = \tfrac{c}{2}$ in~\eqref{e.19} and combining \eqref{e.13},~\eqref{e.15}, and~\eqref{e.19} we deduce that
\begin{equation}\label{e.H13}
 \tfrac{\delta}{2}( \| \dot{z}_{i}\|_{2}^{2} - \| \dot{z}_{i-1} \|_{2}^{2} ) +\tfrac{c \tau_{k}}{2} \|\dot{z}_{i}\|_{H^{1}}^{2} \leq C \tau_{k} ( \|\dot{z}_{i}\|_{2}^{2} + \| \dot{g}_{i} \|^{2}_{W^{1,p}})
\end{equation}
for some $C>0$ independent of $k$ and $\delta$.

For $i=1$ we consider the conditions
\begin{eqnarray*}
&&\displaystyle \partial_{z} \F (u_{0}, z_{0}) [\varphi]\geq 0 \qquad \text{for every $\varphi \in H^{1}(\Om)$, $\varphi \leq 0$}\,. \\[2mm]
&&\displaystyle \partial_{z} \F(u_{1}^{k}, z_{1}^{k})[ \dot{z}_{1} ] + \delta \| \dot{z}_{1} \|_{2}^{2} = 0\,.
\end{eqnarray*}
Testing the first inequality with~$\varphi = \dot{z}_{1} \leq 0$ and subtracting the second equality we deduce
\begin{displaymath}
\delta \| \dot{z}_{1}\|_{2}^{2} \leq  \partial_{z} \F (u_{0}, z_{0}) [ \dot{z}_{1} ] - \partial_{z} \F(u_{1}^{k}, z_{1}^{k})[ \dot{z}_{1} ] \,.
\end{displaymath}
Arguing as in~\eqref{e.14}--\eqref{e.18} we get
\begin{equation}\label{e.22}
\delta \| \dot{z}_{1} \|_{2}^{2} + c \tau_{k} \| \dot{z}_{1} \|^{2}_{H^{1}} \leq C \tau_{k} ( \| \dot{z}_{1} \|_{2} + \| \dot{g}_{1} \|_{W^{1,p}} )^{2}\,.
\end{equation}

Let us now fix $t\in (0, T]$ and let $\bar{i}\in \{2, \ldots, k\}$ be such that $t\in [t^{k}_{\bar{i}-1} , t^{k}_{\bar{i}})$. Summing up \eqref{e.H13} for $i=2,\ldots, \bar{i}$ and~\eqref{e.22} and dividing by~$c\in(0,1)$ we get that
\begin{equation}\label{e.H14}
\begin{split}
\delta \| \dot{z}_{k}^{\delta} (t) \|_{2}^{2} + \tau_{k} \sum_{i=1}^{\bar{i}} \| \dot{z}_{i} \|_{H^{1}}^{2} & \leq C \int_{0}^{t_{k}(t)} \| \dot{z}^{\delta}_{k} (\tau) \|_{2}^{2} \,\di \tau + C\int_{0}^{t_{k}(t)} \| \dot{g} (\tau) \|_{W^{1,p}}^{2}\, \di \tau
\\
&
\leq C \Big( 1 +  \int_{0}^{t_{k}(t)} \| \dot{z}^{\delta}_{k} (\tau) \|_{2}^{2} \,\di \tau \Big)  \,.
\end{split}
\end{equation}
Hence, Gronwall inequality implies that
\begin{equation}\label{e.H15}
\delta \| \dot{z}_{k}^{\delta} (t) \|_{2}^{2} \leq C e^{\frac{C}{\delta} t_{k}(t)} \,.
\end{equation}
Inequality~\eqref{e.H11} follows from~\eqref{e.H15} simply by multiplying by~$\delta$ and taking the square root. Finally,~\eqref{e.H12} is a consequence of~\eqref{e.H11} and of~\eqref{e.H14}.
\end{proof}

\begin{lemma}\label{l.length}
There exists $C>0$ independent of~$k$ and $\delta$ such that
\begin{displaymath}
\int_{0}^{T} \| \dot{z}^{\delta}_{k} (\tau) \|_{H^{1}}\, \di \tau \leq C \,.
\end{displaymath}

\begin{proof}
We use here the same notation as in the proof of Lemma~\ref{l.H1H1}. Arguing as in~\eqref{e.13}, for $i\geq2$ we get that
\begin{equation}\label{e.13.1}
 \partial_{z} \F ( u^{k}_{i-1}, z^{k}_{i-1} )  [ \dot{z}_{i} ] - \partial_{z} \F(u^{k}_{i}, z^{k}_{i})[\dot{z}_{i}] \geq \delta \int_{\Om} ( \dot{z}_{i} - \dot{z}_{i-1} ) \dot{z}_{i}\, \di x \geq \delta \| \dot{z}_{i}\|_{2} ( \| \dot{z}_{i}\|_{2} - \| \dot{z}_{i-1} \|_{2} ) \,,
\end{equation}
where the last step is due to Cauchy and triangle inequality. In order to estimate the left-hand side of~\eqref{e.13.1} we proceed as in~\eqref{e.14}-\eqref{e.18}, obtaining
\begin{equation}\label{e.14.1}
\begin{split}
\delta  \| \dot{z}_{i}\|_{2} ( \| \dot{z}_{i}\|_{2} - \| \dot{z}_{i-1} \|_{2} ) + c \tau_{k} \| \dot{z}_{i} \|_{H^{1}}^{2} & \leq C \tau_{k} (  \|\dot{z}_{i}\|_{\lambda}^{2} + \| \dot{g}_{i} \|^{2}_{W^{1,p}}) 
\end{split}
\end{equation}
for some positive $c, C$ independent of~$i$,~$k$, and~$\delta$ and some $\lambda\in (2, +\infty)$. By interpolation and Sobolev inequality we deduce from~\eqref{e.14.1} that
\begin{equation}\label{e.19.1}
\begin{split}
\delta  \| \dot{z}_{i}\|_{2} ( \| \dot{z}_{i}\|_{2} - \| \dot{z}_{i-1} \|_{2} ) + c \tau_{k} \| \dot{z}_{i} \|_{H^{1}}^{2} & \leq  \eps \tau_{k} \| \dot{z}_{i}\|_{H^{1}}^{2} + C_{\eps} \tau_{k} \|\dot{z}_{i}\|_{1}^{2} + C\tau_{k} \| \dot{g}_{i} \|^{2}_{W^{1,p}}
\\
&
\leq \eps \tau_{k} \| \dot{z}_{i}\|_{H^{1}}^{2} + C_{\eps} \tau_{k} \|\dot{z}_{i}\|_{1} \|\dot{z}_{i}\|_{2} + C\tau_{k} \| \dot{g}_{i} \|^{2}_{W^{1,p}}
\end{split}
\end{equation}
for $\eps, C_{\eps} >0$. Rewriting~\eqref{e.19.1} for $\eps< c/2$ and using the inequality $\| \dot{z}_{i} \|_{2} \leq \| \dot{z}_{i} \|_{H^{1}}$ we get
\begin{equation}\label{e.20}
 \delta \| \dot{z}_{i}\|_{2} ( \| \dot{z}_{i}\|_{2} - \| \dot{z}_{i-1} \|_{2} ) + \frac{c \tau_{k}}{4} \| \dot{z}_{i} \|_{2}^{2} +\frac{c \tau_{k}}{4} \| \dot{z}_{i} \|_{H^{1}}^{2} \leq  C \tau_{k} \|\dot{z}_{i}\|_{1} \|\dot{z}_{i}\|_{2} + C\tau_{k} \| \dot{g}_{i} \|^{2}_{W^{1,p}} \,.
\end{equation} 
Multiplying~\eqref{e.20} by 2 and dividing by~$\delta$ we finally obtain an inequality of the form
\begin{displaymath}
2a_{i} (a_{i} - a_{i-1}) + 2\gamma a_{i}^2 + b_i^2 \leq c_i^2 + 2a_i d_i
\end{displaymath}
for $a_{i}\coloneq \| \dot{z}_{i} \|_{2}$, $b_{i} \coloneq \big (\tfrac{c \tau_{k}}{2\delta} \big)^{\frac{1}{2}}\| \dot{z}_{i}\|_{H^{1}}$, $c_{i}\coloneq \big (\frac{2C \tau_{k}}{\delta}\big)^{\frac{1}{2}} \| \dot{g}_{i}\|_{W^{1,p}}$, $d_{i} \coloneq \frac{C\tau_{k}}{\delta} \| \dot{z}_{i} \|_{1}$, and $\gamma\coloneq \frac{c \tau_{k}}{4\delta}$. Applying now the discrete Gronwall inequality of~\cite[Lemma~5.9]{MR3021776} and performing the computation contained in~\cite[Proposition~2.8]{MR3454016} we eventually deduce that
\begin{equation}\label{e.21}
\sum_{i=2}^{k} \tau_{k} \|\dot{z}_{i}\|_{H^{1}} \leq C \Big( T + \delta \| \dot{z}_{1} \|_{2} + \sum_{i=2}^{k}  \tau_k \| \dot{g}_{i} \|_{W^{1,p}} + \sum_{i=2}^{k} \tau_{k} \|\dot{z}_{i}\|_{1}\Big)\,.
\end{equation}

An estimate for~$\| \dot{z}_{1}\|_{H^{1}}$ follows directly from~\eqref{e.22} and an application of Sobolev and interpolation inequalities, so that
\begin{equation}\label{e.22.1}
 \tau_{k} \| \dot{z}_{1} \|_{H^{1}} \leq C \tau_{k} ( \| \dot{z}_{1} \|_{1} + \| \dot{g}_{1} \|_{W^{1,p}} ) \,.
\end{equation}
Hence, summing up~\eqref{e.21} and~\eqref{e.22.1} and applying H\"older inequality to $\dot{g}_{i}$ we obtain
\begin{equation}\label{e.24}
\begin{split}
\int_{0}^{T}  \| \dot{z}^{\delta}_{k} (\tau) \|_{H^{1}}\, \di s &  \leq C \Big( 1 + \delta \| \dot{z}_{1} \|_{2} + \int_{0}^{T} \| \dot{g} (\tau)\|_{W^{1,p}}^{2}\, \di \tau + \sum_{i=1}^{k} \int_{\Om} z^{k}_{i-1} - z^{k}_{i} \,\di x \Big) 
\\
&
= C \Big( 1 + \delta \| \dot{z}_{1} \|_{2} + \int_{0}^{T} \| \dot{g} (\tau)\|_{W^{1,p}}^{2}\, \di \tau + \int_{\Om} z_{0} - z^{k}_{k} \, \di x\Big)
\\
& 
\leq C \Big( 1 + \delta \| \dot{z}_{1} \|_{2} + \int_{0}^{T} \| \dot{g} (\tau)\|_{W^{1,p}}^{2}\, \di \tau \Big)\,.
\end{split}
\end{equation}

We conclude by observing that in view of~\eqref{e.H11} for $t= t^{k}_{1} = \tau_{k}$ we have that $\delta \| \dot{z}_{1} \|_{2} \leq C e^{\frac{C}{\delta} \tau_{k}}$ for some $C>0$ independent of~$k$ and~$\delta$. Hence, for $\tau_{k} \leq \delta$ we infer that $\delta  \| \dot{z}_{1} \|_{2}$ is bounded and, in view of~\eqref{e.24},
\begin{displaymath}
\int_{0}^{T}  \| \dot{z}^{\delta}_{k} (\tau) \|_{H^{1}}\, \di \tau  \leq C \Big( 1 + \int_{0}^{T} \| \dot{g} (\tau)\|_{W^{1,p}}^{2}\, \di \tau \Big)\,,
\end{displaymath}
which implies the thesis.
\end{proof}
\end{lemma}

In the next proposition we provide a discrete energy inequality.

\begin{proposition}\label{prop3}
Let $\delta>0$ be fixed. Then, the following facts hold:
\begin{itemize}
\item[$(a)$] The sequences $\bar{u}^{\delta}_{k}$ and $\ubar{u}^{\delta}_{k}$ are bounded in $L^{\infty}([ 0,T] ; H^{1}(\Om;\R^{2}))$;
\item[$(b)$] For every $t \in [0,T]$  we have 
\begin{displaymath}
\bar{u}^{\delta}_k (t) \in \argmin\, \{ \E ( u , \bar{z}^{\delta}_k (t) ) :  u \in H^{1}(\Om;\R^{2}), \, u = g ( t_k(t) ) \text{ on $\partial_D \Om$}  \} \,;
\end{displaymath}
\item[$(c)$] There exists a constant $C>0$ such that for every $k\in\mathbb{N}$ and every $t\in[0,T]$
\begin{equation}\label{energyEstimate}
\begin{split}
\F(\bar{u}^{\delta}_{k} (t) , \bar{z}^{\delta}_{k} (t ) ) & \leq \F(u_{0} , z_{0} ) - \tfrac{1}{2\delta} \int_{0}^{t_{k}(t)}  \!\!\!\!\!  |\partial^-_z \F |^{2} ( \bar{u}^{\delta}_{k} (\tau) , z^{\delta}_{k} (\tau)) \, \di \tau - \tfrac{\delta}{2}\int_{0}^{ t_{k}(t)} \!\!\!\!\! \|\dot{z}^{\delta}_{k} (\tau) \|^{2}_{2} \, \di \tau
\\
&
\qquad + \int_{0}^{t_{k}(t)}\!\!\!\!\!\mathcal{P} ( \ubar{u}^{\delta}_{k}(\tau) , \ubar{z}^{\delta}_{k}(\tau) , \dot g(\tau) ) \, \di \tau + R_k \,,
\end{split}
\end{equation}
where $R_k \geq 0$ is such that $R_k \to 0$ as $k\to\infty$.
\end{itemize}
\end{proposition}

\begin{proof} We notice that condition~(b) is true by definition of~$\bar{u}^{\delta}_{k}$ and~$\bar{z}^{\delta}_{k}$ and by~\eqref{e.3}.

Let us prove the energy estimate $(c)$. For $k\in\mathbb{N}$, $t\in (0,T]$ fixed, let $i\in\{1,\ldots,k\}$ be such that $t\in(t^{k}_{i-1},t^{k}_{i}]$. By convexity of $z\mapsto\F(u^{k}_{i} ,  z)$ and by applying Proposition~\ref{p.2} we get
\begin{equation}\label{9.2}
\begin{split}
\F(u^{k}_{i} , z^{k}_{i-1})& \geq \F ( u^{k}_{i} , z^{k}_{i} ) + \partial_{z} \F ( u^{k}_{i} , z^{k}_{i} ) [ z^{k}_{i-1} - z^{k}_{i} ]
= \F ( u^{k}_{i} , z^{k}_{i} ) - \tau_{k} \partial_{z} \F ( u^{k}_{i} , z^{k}_{i} ) [ \dot{z}^{\delta}_{k} (t) ]
\\
&
\vphantom{\int} = \F ( u^{k}_{i} , z^{k}_{i} ) + \tau_{k} |\partial^-_z \F | ( u^{k}_{i} , z^{k}_{i} ) \| \dot{z}^{\delta}_{k} (t) \|_{2}
\\
&
= \F ( u^{k}_{i} , z^{k}_{i} ) + \tfrac{1}{2\delta} \int_{t^{k}_{i-1}}^{t^{k}_{i}} \!\!\!\! |\partial^-_z \F |^{2} ( \bar{u}^{\delta}_{k} (\tau) , \bar{z}^{\delta}_{k} (\tau) ) \,\di \tau + \tfrac{\delta}{2} \int_{t^{k}_{i-1}}^{t^{k}_{i}} \!\!\!\! \| \dot{z}^{\delta}_{k} (\tau) \|_{2}^{2} \, \di \tau \,.
\end{split}
\end{equation}

We now have to pass from $u^{k}_{i}$ to~$u^{k}_{i-1}$ in the left-hand side of~\eqref{9.2}. We perform this passage in two steps, first moving from~$u^{k}_{i}$ to~$u^{k}_{i,1}$ and then from~$u^{k}_{i,1}$ to~$u^{k}_{i-1}$. For the second step we can simply rely on the minimality~\eqref{minu} of~$u^{k}_{i,1}$. For the first step, instead, we make use of the regularity estimate of Lemma~\ref{l.regLp}. Hence, by convexity of $u\mapsto \F (u, z)$ we have that
\begin{equation} \label{e.55}
\begin{split}
\F(u^{k}_{i}, z^{k}_{i-1} ) & \leq \F (u^{k}_{i,1} , z^{k}_{i-1}) -\partial_{u} \F( u_{i}^{k}, z^{k}_{i-1}) [u^{k}_{i,1} - u^{k}_{i}]
\\
&
\leq \F (u^{k}_{i-1} + g(t^{k}_{i}) - g(t^{k}_{i-1}) , z^{k}_{i-1} ) -\partial_{u} \F( u_{i}^{k}, z^{k}_{i-1}) [u^{k}_{i,1} - u^{k}_{i}]
\\
&
= \F (u^{k}_{i-1} , z^{k}_{i-1} ) + \int_{t^{k}_{i-1}}^{t^{k}_{i}}  \partial_{u} \F\Big(u^{k}_{i-1} + (s-t^{k}_{i-1})\frac{g(t^{k}_{i}) - g(t^{k}_{i-1})}{\tau_{k}} , z^{k}_{i-1} \Big) \Big[ \frac{g(t^{k}_{i}) - g(t^{k}_{i-1})}{\tau_{k}} \Big] \, \di s  
\\
&
\qquad - \partial_{u} \F( u_{i}^{k}, z^{k}_{i-1}) [u^{k}_{i,1} - u^{k}_{i}] \,.
\end{split}
\end{equation}
We estimate the right-hand side of~\eqref{e.55}. By minimality of~$u^{k}_{i}$ we have that
\begin{displaymath}
\begin{split}
\partial_{u} \F ( u_{i}^{k} , z^{k}_{i-1}) [u^{k}_{i,1} - u^{k}_{i} ] & = \partial_{u} \F ( u^{k}_{i} , z^{k}_{i} ) [ u^{k}_{i,1} - u^{k}_{i} ] + \int_{\Om} \mu ( h( z^{k}_{i-1} ) - h( z^{k}_{i} ))  \strain_{d} ( u^{k}_{i} ){\,:\,} \strain_{d}(u^{k}_{i,1} - u^{k}_{i}) \, \di x 
\\
&
\qquad  + \int_{\Om} \kappa ( h( z^{k}_{i-1} ) - h( z^{k}_{i} )) \strain_{v}^{+} (u^{k}_{i}){\, : \,} \strain_{v}(u^{k}_{i,1} - u^{k}_{i}) \, \di x
\\
&
=\int_{\Om} \mu ( h( z^{k}_{i-1} ) - h( z^{k}_{i} ))  \strain_{d} ( u^{k}_{i} ){\,:\,} \strain_{d}(u^{k}_{i,1} - u^{k}_{i}) \, \di x 
\\
&
\qquad  + \int_{\Om} \kappa ( h( z^{k}_{i-1} ) - h( z^{k}_{i} )) \strain_{v}^{+} (u^{k}_{i}){\, : \,} \strain_{v}(u^{k}_{i,1} - u^{k}_{i}) \, \di x \,.
\end{split}
\end{displaymath}
Hence, choosing $r>2$ as in Lemma~\ref{l.regLp}, from the regularity of~$h$ and Sobolev embedding we deduce that
\begin{equation}\label{e.5.1}
| \partial_{u} \F ( u_{i}^{k} , z^{k}_{i-1}) [u^{k}_{i,1} - u^{k}_{i} ] |  \leq C \| z^{k}_{i-1} - z^{k}_{i}\|_{\nu}  \| u^{k}_{i,1} - u^{k}_{i} \|_{W^{1,r}} \leq C \| z^{k}_{i} - z^{k}_{i-1} \|_{H^{1}}^{2}\,,
\end{equation}
where~$C$ is a positive constant independent of~$k$ and~$\delta$ and $\tfrac{1}{\nu} + \tfrac{2}{r} = 1$. In a similar way, by Lemma~\ref{l.HMWw}, we obtain that for every $s\in [t^{k}_{i-1}, t^{k}_{i}]$
\begin{displaymath}
\begin{split}
\partial_{u} \F  \Big(u^{k}_{i-1} & + (s - t^{k}_{i-1}) \frac{g(t^{k}_{i}) - g(t^{k}_{i-1})}{\tau_{k}} , z^{k}_{i-1} \Big) \Big[ \frac{g(t^{k}_{i}) - g(t^{k}_{i-1})}{\tau_{k}} \Big] 
\\
&
\leq \partial_{u} \F (u^{k}_{i-1} , z^{k}_{i-1}) \Big[ \frac{g(t^{k}_{i}) - g(t^{k}_{i-1})}{\tau_{k}} \Big] + \frac{C}{ \tau_{k}} \| g(t^{k}_{i}) - g(t^{k}_{i-1})\|_{H^{1}}^{2}\,,
\end{split}
\end{displaymath}
which yields
\begin{equation}\label{e.5.2}
\begin{split}
\int_{t^{k}_{i-1}}^{t^{k}_{i}} \partial_{u} \F  \Big(u^{k}_{i-1} & + (s - t^{k}_{i-1}) \frac{g(t^{k}_{i}) - g(t^{k}_{i-1})}{\tau_{k}} , z^{k}_{i-1} \Big) \Big[ \frac{g(t^{k}_{i}) - g(t^{k}_{i-1})}{\tau_{k}} \Big] \, \di s
\\
&
\leq \int_{t^{k}_{i-1}}^{t^{k}_{i}} \partial_{u} \F (u^{k}_{i-1}, z^{k}_{i-1})[\dot{g}(\tau)]\, \di \tau + C  \| g(t^{k}_{i}) - g(t^{k}_{i-1})\|_{H^{1}}^{2}
\\
& = \int_{t^{k}_{i-1}}^{t^{k}_{i}} \mathcal{P} (u^{k}_{i-1}, z^{k}_{i-1}, \dot{g}(\tau)) \, \di \tau + C  \| g(t^{k}_{i}) - g(t^{k}_{i-1})\|_{H^{1}}^{2}\,.
\end{split}
\end{equation}

Combining~\eqref{e.55}-\eqref{e.5.2} we obtain~\eqref{energyEstimate} with
\begin{displaymath}
R_k = C \sum_{i=1}^{k} \| z^{k}_{i} - z^{k}_{i-1} \|_{H^{1}}^{2} + \| g(t^{k}_{i}) - g(t^{k}_{i-1})\|_{H^{1}}^{2}\,.
\end{displaymath}
In particular, $R_k\to 0$ as $k\to\infty$ thanks to the bound~\eqref{e.H12} of Lemma~\ref{l.H1H1} and to the regularity of~$g$.
\end{proof}

\section{Viscous and quasistatic evolutions} \label{s.evolution}

In this section we show the existence of a viscous evolution in the sense of Definition \ref{d.Evolution} for every $\delta>0$. Furthermore, we study the limit as $\delta\to 0$ of the above evolutions proving that, in a suitable time-reparametrized setting, they converge to a vanishing viscosity evolution in the sense of Definition~\ref{d.vanevolution}.

\begin{theorem}\label{t.viscous}
Let $g\in H^{1}([0,T] ; W^{1,p}(\Om;\R^{2}))$ for some $p \in (2,+\infty)$ and $(u_{0}, z_{0}) \in H^{1}(\Om;\R^{2}) \times H^{1}(\Om;[0,1])$ be such that $u_0 = g(0)$ on $\partial_{D}\Om$ and \eqref{equilibriumu0}--\eqref{equilibriumz0} hold.
For every $\delta>0$ and every $k\in\mathbb{N}$ let $z^{\delta}_{k}$,~$\bar{z}^{\delta}_{k}$,~$\underline{z}^{\delta}_{k}$,~$\bar{u}^{\delta}_{k}$, and~$\underline{u}^{\delta}_{k}$ be as in~\eqref{int1}-\eqref{int3} with $\bar{u}^{\delta}_{k} (0) = \underline{u}^{\delta}_{k} (0) = u_{0}$ and $z^{\delta}_{k} (0) = \bar{z}^{\delta}_{k} (0) = \underline{z}^{\delta}_{k} (0) = z_{0}$. Then, there exists a viscous evolution $(u_{\delta}, z_{\delta}) \colon [0,T] \to H^{1}(\Om;\R^{2})\times H^{1}(\Om)$  with initial condition~$(u_{\delta}(0) , z_{\delta}(0)) = (u_{0} , z_0)$ and such that, up to a subsequence, $u_{k}^{\delta} (t) \to u_{\delta}(t)$ in $H^{1}(\Om;\R^{2})$ for every $t\in[0,T]$ and $z_{k}^{\delta} \rightharpoonup z_{\delta}$ weakly in~$H^{1}([0,T]; H^{1}(\Om))$.

Moreover, it holds
\begin{eqnarray}
&& \displaystyle \sup_{\delta>0}\, \int_{0}^{T} \|\dot{z}_{\delta}(\tau) \|_{H^{1}}\, \di \tau < +\infty \label{e.boundviscosity} \,, \\[2mm]
&& \displaystyle \delta \| \dot{z}_{\delta} (t) \|_{2} = |\partial_{z}^{-} \F| ( u_{\delta}(t), z_{\delta}(t)) \qquad \text{for a.e.~$t\in[0,T]$} \label{e.ODEviscosity} \,.
\end{eqnarray}
\end{theorem}

\begin{proof}
Fix $\delta>0$. By Lemma~\ref{l.H1H1} we infer that there exists $z_{\delta} \in H^{1} ( [0,T] ; H^{1} (\Om) )$ such that, up to a subsequence, $z^{\delta}_{k} \rightharpoonup z_{\delta}$ weakly in~$H^{1} ( [0,T]; H^{1} (\Om))$ and $z^{\delta}_{k}(t) \rightharpoonup z_{\delta}(t)$ weakly in~$H^{1} (\Om)$ for every~$t\in[0,T]$. By regularity of~$z^{\delta}_{k}$ we also deduce that $\bar{z}^{\delta}_{k}(t)$ and~$\underline{z}^{\delta}_{k} (t)$ converge to~$z_{\delta}(t)$ weakly in~$H^{1}(\Om)$ for $t\in[0,T]$. Furthermore, inequality~\eqref{e.boundviscosity} follows from Lemma~\ref{l.length} and lower semicontinuity.

For every $t\in[0,T]$ we have that, by~(a) of Proposition~\ref{prop3}, there exists $u_{\delta} (t) \in H^{1}(\Om;\R^{2})$ such that, up to a further time-dependent subsequence, $\bar{u}^{\delta}_{k}(t) \rightharpoonup u_{\delta} (t)$ weakly in~$H^{1}(\Om;\R^{2})$. From Lemma~\ref{lemma1} and from (b) of Proposition~\ref{prop3} we deduce that $\bar{u}^{\delta}_{k}(t) \to u_{\delta} (t)$ in~$H^{1}(\Om;\R^{2})$ and that
\begin{displaymath}
u_{\delta} (t) = \argmin \, \{ \E (u, z_{\delta}(t)) : \, u \in H^{1} ( \Om ; \R^{2} ), \, u = g(t) \text{ on $\partial_{D}\Om$} \} \,.
\end{displaymath}
By uniqueness of minimizer, the above convergence holds for every $t\in[0,T]$ along the whole sequence~$k$. Arguing in the same way for the sequence~$\ubar{u}^{\delta}_{k}(t)$ we get that~$\ubar{u}^{\delta}_{k}(t) \to u_{\delta} (t)$ in~$H^{1}(\Om;\R^{2})$.

It remains to show the energy balance~(d) of Definition~\ref{d.Evolution}. To do this, we first pass to the liminf as $k\to\infty$ in the discrete inequality~\eqref{energyEstimate} of Proposition~\ref{prop3}. Combining the above convergences, the lower semicontinuity of~$\F$, Fatou lemma, Lemma~\ref{prop1}, and the regularity of~$g$, we get the lower inequality
\begin{displaymath}
\begin{split}
\F ( u_{\delta}(t), z_{\delta}(t))  \leq & \ \F(u_0, z_0) - \tfrac{1}{2\delta} \int_{0}^{t} |\partial_{z}^{-} \F |^{2} (u_{\delta} (\tau) , z_{\delta} (s)) \, \di \tau - \tfrac{\delta}{2} \int_{0}^{t} \| \dot{z}_{\delta} (\tau) \|_{2}^{2}\, \di \tau 
\\
&
+ \int_{0}^{t} \mathcal{P} (u_{\delta} (\tau) , z_{\delta}(\tau), \dot{g}(\tau))\, \di \tau \,.
\end{split}
\end{displaymath}

As for the opposite inequality, we notice that $z_{\delta} \in H^{1}([0,T]; H^{1}(\Om))$, so that by chain rule (see, e.g.,~\cite[Corollary~2.9]{MR3021776}), we have
\begin{equation}\label{e.25}
\begin{split}
\F(u_{\delta} (t), z_{\delta}(t)) = \F(u_0, z_0) + \int_{0}^{t} \partial_{z} \F( u_{\delta} (\tau), z_{\delta} (\tau)) [\dot{z}_{\delta}(\tau)] \,\di s + \int_{0}^{t} \mathcal{P} (u_{\delta}(\tau), z_{\delta}(\tau), \dot{g}(\tau))\, \di \tau \,.
\end{split}
\end{equation}
In view of Lemma~\ref{SlopeLemma}, we can estimate~\eqref{e.25} from below with
\begin{equation}\label{e.26}
\F(u_{\delta} (t), z_{\delta}(t)) \geq \F(u_0, z_0) - \int_{0}^{t} |\partial_{z}^{-} \F| ( u_{\delta} (\tau), z_{\delta} (\tau)) \| \dot{z}_{\delta}(\tau) \|_{2} \,\di \tau + \int_{0}^{t} \mathcal{P} (u_{\delta}(\tau), z_{\delta}(\tau), \dot{g}(\tau))\, \di \tau \,.
\end{equation}
The energy equality (d) of Definition~\ref{d.Evolution} follows by Young inequality. Finally, we notice that~\eqref{e.26} and the energy balance imply~\eqref{e.ODEviscosity}, and the proof is thus concluded.
\end{proof}

We can now conclude showing that in the limit as $\delta\to 0$ the viscous evolutions determined in Theorem~\ref{t.viscous} converge to a vanishing viscosity evolution (see Definition~\ref{d.vanevolution}). In order to do this, we first introduce an arc length reparametrization of time which makes~$z_{\delta}$ Lipschitz continuous. For every $t\in[0,T]$ we define
\begin{displaymath}
\sigma_{\delta}(t) \coloneq t + \int_{0}^{t} \|\dot{z}_{\delta} (\tau) \|_{H^{1}}\,\di \tau \,.
\end{displaymath}
Clearly, $\sigma_{\delta} \colon [0,T]\to [0, \sigma_{\delta}(T)]$ is continuous and strictly increasing, and therefore invertible. We denote its inverse with~$t_{\delta}$. For $\sigma \in [0,\sigma_{\delta}(T)]$ we define
\begin{equation}\label{e.31}
\tilde{u}_{\delta} (s) \coloneq u_{\delta}( t_{\delta} (s)) \qquad \text{and} \qquad \tilde{z}_{\delta}(s) \coloneq z_{\delta} (t_{\delta}(s))\,.
\end{equation}
By definition of~$\sigma_{\delta}$, of~$t_{\delta}$, and of~$\tilde{z}_{\delta}$, we have that
\begin{equation}\label{e.32}
t'_{\delta}(s) + \| \tilde{z}_{\delta}'(s)\|_{H^{1}} =1 \qquad\text{for a.e.~$s \in [0, \sigma_{\delta}(T)]$}\,.
\end{equation}
In view of~\eqref{e.boundviscosity}, there exists $S \in (0, +\infty)$ such that $\sigma_{\delta}(T) \leq S$ for every $\delta>0$. Hence, we may extend $t_{\delta}$, $\tilde{u}_{\delta}$, and~$\tilde{z}_{\delta}$ in a constant way in~$[\sigma_{\delta}(T), S]$, so that $t_{\delta} \in W^{1,\infty}(0,S)$ and $\tilde{z}_{\delta} \in W^{1,\infty}([0,S]; H^{1}(\Om))$. Finally, we notice that
\begin{equation}\label{e.30}
\tilde{u}_{\delta}(s) = \argmin \, \{ \E(u, \tilde{z}_{\delta}(s)) : \, u\in H^{1}(\Om;\R^{2}),\, u=g(t_{\delta}(s)) \text{ on $\partial_{D}\Om$}\} \,.
\end{equation}

\begin{theorem}\label{t.vanevolution}
Let $g\in H^{1}([0,T]; W^{1,p}(\Om; \R^{2}))$ for some $p \in (2,+\infty)$ and $(u_{0}, z_{0}) \in H^{1}(\Om;\R^{2}) \times H^{1}(\Om;[0,1])$ be such that $u_0 = g(0)$ on $\partial_{D}\Om$ and \eqref{equilibriumu0}--\eqref{equilibriumz0} hold. For every $\delta>0$ let $t_{\delta}$, $\tilde{u}_{\delta}$, and $\tilde{z}_{\delta}$ be as in~\eqref{e.31}. Then, there exists a vanishing viscosity evolution~$(t, u, z) \colon [0,S] \to [0,T] \times H^{1}(\Om;\R^{2})\times H^{1}(\Om)$ with initial condition~$(u_{0}, z_{0})$ and boundary datum $g$ such that, up to a subsequence, $t_{\delta}\rightharpoonup t$ weakly$^{*}$ in~$W^{1,\infty}(0,S)$, $\tilde{z}_{\delta} \rightharpoonup z$ weakly$^{*}$ in~$W^{1,\infty}([0,S]; H^{1}(\Om))$, and $u_{\delta}(s) \to u(s)$ in $H^{1}(\Om;\R^{2})$ for every $s \in [0,S]$.
\end{theorem}

\begin{proof}
By \eqref{e.32} and Definition~\ref{d.Evolution},~$t_{\delta}$ and~$\tilde{z}_{\delta}$ are bounded in~$W^{1,\infty}(0,S)$ and $W^{1,\infty}([0,S]; H^{1}(\Om))$, respectively. Hence, up to subsequence, we have that $t_{\delta} \rightharpoonup t$ weakly$^{*}$ in~$W^{1,\infty}(0,S)$ and $\tilde{z}_{\delta} \rightharpoonup z$ weakly$^{*}$ in $W^{1,\infty}([0,S]; H^{1}(\Om))$, where $z$ is decreasing w.r.t.~$s$ and satisfies $0 \leq z( s ) \leq 1$ for every $s \in [0,S]$. As for $\tilde{u}_{\delta}$, from~\eqref{e.30} we infer that~$\tilde{u}_{\delta}$ is bounded in $L^{\infty}([0,S];H^{1}(\Om;\R^{2}))$. Thus, for every $s \in [0,S]$ there exists~$u(s) \in H^{1}(\Om;\R^{2})$ such that, up to an $s$-dependent subsequence, $\tilde{u}_{\delta} (s) \rightharpoonup u(s)$ weakly in~$H^{1}(\Om;\R^{2})$. Lemma~\ref{lemma1} implies that~$u(s)$ satisfies~(d) of Definition~\ref{d.vanevolution} for every~$s \in [0,S]$ and, by uniqueness of minimizers, $\tilde{u}_{\delta}(s) \to u(s)$ in~$H^{1}(\Om;\R^{2})$ for every $s \in [0,S]$ along a unique subsequence independent of~$s$.

By lower semicontinuity,~\eqref{e.32} implies that (c) of Definition~\ref{d.vanevolution} holds. By a change of variable in the energy balance~\eqref{e.enbaldelta} we obtain
\begin{equation}\label{e.33}
\begin{split}
\F(\tilde{u}_{\delta}(s) , \tilde{z}_{\delta}(s)) = & \ \F(u_0, z_0) - \int_{0}^{s} |\partial_{z}^{-} \F| (\tilde{u}_{\delta}(\sigma) , \tilde{z}_{\delta}(\sigma)) \| \tilde{z}'_{\delta} (\sigma) \|_{2} \, \di\sigma 
\\
&
+ \int_{0}^{s} \mathcal{P}(\tilde{u}_{\delta}(\sigma), \tilde{z}_{\delta}(\sigma), \dot{g}(t_{\delta}(\sigma))) \, t_{\delta}'(\sigma)\, \di \sigma \,.
\end{split}
\end{equation}
We now pass to the liminf as $\delta\to 0$ in~\eqref{e.33}. By Lemma~\ref{lsc-F} and continuity of the power~$\mathcal{P}$ we have that
\begin{displaymath}
\begin{split}
\F( u (s) , z (s)) \leq & \ \F(u_0, z_0) - \liminf_{\delta\to 0} \int_{0}^{s} |\partial_{z}^{-} \F| (\tilde{u}_{\delta}(\sigma) , \tilde{z}_{\delta}(\sigma)) \| \tilde{z}'_{\delta} (\sigma) \|_{2} \, \di\sigma 
\\
&
+ \int_{0}^{s} \mathcal{P}(u (\sigma), z( \sigma), \dot{g}(t (\sigma))) \, t ' (\sigma)\, \di \sigma \,.
\end{split}
\end{displaymath}
Applying~\cite[Theorem~3.1]{Balder_RCMP85} we deduce the lower semicontinuity of the second integral on the right-hand side of~\eqref{e.33}, so that we conclude the lower energy inequality
\begin{displaymath}
\F( u (s) , z (s)) \leq  \F(u_0, z_0) -  \int_{0}^{s} |\partial_{z}^{-} \F| (u (\sigma) , z (\sigma)) \| z' (\sigma) \|_{2} \, \di\sigma 
+ \int_{0}^{s} \mathcal{P}(u (\sigma), z (\sigma), \dot{g}(t (\sigma))) \, t ' (\sigma)\, \di \sigma \,.
\end{displaymath}
The opposite inequality is, as in the proof of Theorem~\ref{t.viscous}, a consequence of the chain rule. Indeed, being $z\in W^{1,\infty}([0,S]; H^{1}(\Om))$, we can write for every $s\in[0,S]$
\begin{equation}\label{e.chainrule}
\F( u (s) , z (s)) =  \F(u_0, z_0) + \int_{0}^{s} \partial_{z} \F ( u (\sigma) , z (\sigma)) [ z' (\sigma) ] \, \di \sigma 
+ \int_{0}^{s} \mathcal{P}(u (\sigma), z (\sigma), \dot{g}(t (\sigma))) \, t ' (\sigma)\, \di \sigma \,.
\end{equation}
Since $ \| z'(\sigma) \|_{2} \leq 1$ for a.e.~$\sigma\in[0,S]$, the previous equality implies the energy balance~(e) of Definition~\ref{d.vanevolution}, and the proof is thus concluded.
\end{proof}

We now collect some properties satisfied by the evolutions $(t, u, z)$ constructed in Theorem~\ref{t.vanevolution}. In order to do this, we first recall a result on the representation of linear functional on~$H^{1}(\Om)$ (see, for instance,~\cite[Lemma~A.3 and Corollary~A.4]{Negri_ACV} and~\cite{MR1327457}).

\begin{lemma}\label{l.representation}
Let $\zeta \in (H^{1}(\Om))'$ be such that
\begin{displaymath}
\sup \, \{ \left \langle \zeta, \varphi \right\rangle : \, \varphi \in H^{1}(\Om),\, \varphi \geq 0,\, \|\varphi\|_{2} \leq 1\} < +\infty\,.
\end{displaymath}
Then, $\zeta$ is a finite Radon measure whose positive part $\zeta_{+} \in L^{2}(\Om)$. Moreover, if
\begin{displaymath}
\bar{\varphi} \in \argmax \,\{  \left\langle \zeta, \varphi \right\rangle : \, \varphi \in H^{1}(\Om),\, \varphi \geq 0,\, \|\varphi\|_{2} \leq 1\}\,,
\end{displaymath}
then $\zeta_{+} = \bar{\varphi} \| \zeta_{+}\|_{2}$.
\end{lemma}

\begin{proposition}\label{p.PDE}
Let $(t, u, z)\colon[0,S]\to [0,T]\times H^{1}(\Om;\R^{2})\times H^{1}(\Om)$ be the vanishing viscosity evolution found in Theorem~\ref{t.vanevolution}, and set $U\coloneq \{ s\in[0,S]: \, t(\cdot) \text{ is constant in a neighborhood of~$s$}\}$. Then,  the following facts hold:
\begin{itemize}
\item[$(a)$] for a.e.~$s\in[0,S]$, $|\partial_{z}^{-}\F|(u(s), z(s)) \| z'(s)\|_{2} = - \partial_{z} \F(u(s), z(s)) [z'(s)]$; 

\item[$(b)$] for a.e.~$s\in[0,S]$,
\begin{displaymath}
z'(s) \| \big( \partial_{z} W(z(s), \strain(u(s))) - \Delta z(s) + f'((z(s)) \big)_{+} \|_{2} =\| z' (s) \|_{2} \big( \partial_{z} W(z(s), \strain(u(s))) - \Delta z(s) + f'((z(s)) \big)_{+} ;
\end{displaymath}

\item[$(c)$] for every $s\in[0,S]\setminus U$, $|\partial_{z}\F| (u(s), z(s)) =0$ and
\begin{equation}\label{e.PDE}
\big( \partial_{z} W(z(s), \strain(u(s))) - \Delta z(s) + f'((z(s)) \big)_{+} = 0 \qquad\text{in $\Om$}\,.
\end{equation}
\end{itemize}
\end{proposition}

\begin{proof}
The equality in~(a) follows from the chain rule~\eqref{e.chainrule} and the energy equality~\eqref{e.viscousenergybalance}. In particular, we deduce that, whenever $\| z'(s)\|_{2} \neq 0$,
\begin{equation}\label{e.z'}
\frac{z'(s)}{\| z' (s) \|_{2}} \in \argmax \, \{ -\partial_{z} \F(u(s), z(s)) [\varphi] : \, \varphi\in H^{1}(\Om), \,\varphi\leq 0,\, \|\varphi\|_{2} \leq 1\}\,.
\end{equation}
Hence, we infer from Lemma~\ref{l.representation} that
\begin{equation}
 \big( \partial_{z} W(z(s), \strain(u(s))) - \Delta z(s) + f'((z(s)) \big)_{+} = \big( \partial_{z} \F(u(s), z(s)) \big)_{+} \in L^{2}(\Om)\,,
 \end{equation}
 \begin{displaymath}
 z'(s) \| \big( \partial_{z} W(z(s), \strain(u(s))) - \Delta z(s) + f'((z(s)) \big)_{+} \|_{2} =\| z' (s) \|_{2} \big( \partial_{z} W(z(s), \strain(u(s))) - \Delta z(s) + f'((z(s)) \big)_{+} ,
\end{displaymath}
where the last equality is trivially extended to the case $\| z'(s)\|_{2} =0$. Thus, also (b) holds.

In order to prove (c) we argue as in~\cite[Theorem~5.4]{MR3454016}. Namely, we show that the set
\begin{displaymath}
A \coloneq \{ s\in (0,S) : \, |\partial_{z}^{-}\F| (u(s), z(s)) >0\}
\end{displaymath}
is contained in~$U$. First, we notice that, in view of Lemma~\ref{prop1} and of the continuity of $s\mapsto (u(s), z(s))$ as a map with values in $H^{1}(\Om;\R^{2})\times H^{1}(\Om)$, the set~$A$ is open in~$(0,S)$. Let $\bar s\in A$ and $C>0$ be such that $|\partial_{z}^{-} \F| (u(\bar s), z(\bar s)) >C >0$. Again by Lemma~\ref{prop1}, there exist $s_{1}< \bar s < s_{2}$ such that $(s_{1}, s_{2}) \subseteq A$ and
\begin{equation}\label{e.bars}
|\partial_{z}^{-}\F|(u(s), z(s)) >C \qquad\text{for every $s\in(s_{1}, s_{2})$}\,.
\end{equation}
Since, by Theorem~\ref{t.vanevolution}, $t_{\delta}(s) \to t(s)$, $ u_{\delta}(s) \to u(s)$ in~$H^{1}(\Om;\R^{2})$, and $z_{\delta}(s) \rightharpoonup z(s)$ weakly in~$H^{1}(\Om)$, the lower semicontinuity of the slope~$|\partial_{z}^{-}\F|$ implies that
\begin{equation}\label{e.bars2}
\liminf_{\delta\to 0} |\partial_{z}^{-}\F|( u_{\delta}(s), z_{\delta}(s)) \geq C\,.
\end{equation}

By definition of~$t_{\delta}$ and by~\eqref{e.ODEviscosity} we have that
\begin{displaymath}
\int_{s_{1}}^{s_{2}} t'_{\delta}(s)\, \di s = \int_{s_{1}}^{s_{2}} \frac{1}{1+ \| z'_{\delta}(s) \|_{H^{1}}}\, \di s \leq \int_{s_{1}}^{s_{2}} \frac{\delta}{\delta + |\partial_{z}^{-}\F| (u_{\delta}(s), z_{\delta}(s))} \, \di s\,. 
\end{displaymath}
Passing to the limsup as $\delta\to 0$ in the previous equality, thanks to the convergence of~$t_{\delta}$ to~$t$ weak$^*$ in~$W^{1,\infty}(0,S)$ and to~\eqref{e.bars2} we obtain
\begin{displaymath}
\int_{s_{1}}^{s_{2}} t'(s)\,\di s = 0\,.
\end{displaymath}
Being~$t$ absolutely continuous, we have that $t$ is constant in the interval $(s_{1}, s_{2})\ni \bar s$. Hence, every element of~$A$ has a neighborhood in which $t(\cdot)$ is constant. Thus, $A\subseteq U$. It follows that $|\partial_{z}^{-}\F|(u(s), z(s)) =0 $ for every $s\in (0,S)\setminus U$. Finally,~\eqref{e.PDE} follows from the inequality
\begin{displaymath}
\partial_{z}\F (u(s), z(s)) [\varphi] \geq 0 \qquad\text{for every $\varphi\in H^{1}(\Om)$ with $\varphi\leq 0$}\,.
\end{displaymath}
\end{proof}

\bibliographystyle{siam}
\bibliography{Almi_19-biblio}

\end{document}